\documentclass[11pt,a4]{article}

\usepackage{enumerate}
\usepackage{latexsym} 
\usepackage{graphics}
\usepackage{graphicx}
\usepackage{amsthm}
\usepackage{amssymb}
\usepackage{amsbsy}
\usepackage{mathrsfs} 
\usepackage{ifpdf}
\usepackage{amstext} 
\newtheorem{defeng}{Definition}[section]
\newtheorem{theorem}[defeng]{Theorem}

\newtheorem{lemma}[defeng]{Lemma}
\newtheorem{proposition}[defeng]{Proposition}

\newtheorem{corollary}[defeng]{Corollary} 
 
\newtheorem*{th:sf}{Theorem \ref{th:sf}}
\newtheorem*{no-phi-c}{Theorem \ref{no-phi-c}} 
\newtheorem*{th:col}{Theorem \ref{th:col}} 
\newtheorem*{th:lbPsi}{Theorem \ref{th:lbPsi}} 
\newtheorem*{th:bounds}{Theorem \ref{th:bounds}}

\newcommand{\sm}{\smallsetminus} 

\ifpdf
\fi

\begin{document}

\title{Isolating highly connected induced subgraphs}

\author{Irena Penev\thanks{Universit\'e de Lyon, LIP, UMR 5668, ENS de Lyon - CNRS - UCBL - INRIA. E-mail: irena.penev@ens-lyon.fr. This work was supported by the LABEX MILYON (ANR-10-LABX-0070) of Universit\'e de Lyon, within the program ``Investissements d'Avenir'' (ANR-11-IDEX-0007) operated by the French National Research Agency (ANR). Partially supported by ANR project Stint under reference ANR-13-BS02-0007.}, St\'ephan Thomass\'e\thanks{CNRS, LIP, ENS de Lyon, INRIA, Universit\'e de Lyon. E-mail: stephan.thomasse@ens-lyon.fr and nicolas.trotignon@ens-lyon.fr. Supported by Labex MILYON and \emph{Agence Nationale de la Recherche} under reference \textsc{anr 10 jcjc 0204 01} and ANR project Stint ANR-13-BS02-0007.}, and Nicolas Trotignon\footnotemark[2]}

\maketitle



\begin{abstract}
 We prove that any graph $G$ of minimum degree greater than
 $2k^2-1$ has a $(k+1)$-connected induced subgraph $H$ such that the
 number of vertices of $H$ that have neighbors outside of $H$ is at
 most $2k^2-1$. This generalizes a classical result of Mader, which states
 that a high minimum degree implies the existence of a highly connected 
 subgraph. We give several variants of our result, and for each of these 
 variants, we give asymptotics for the bounds. We also we compute optimal 
 values for the case when $k=2$. 

 Alon, Kleitman, Saks, Seymour, and Thomassen proved that in
 a graph of high chromatic number, there exists an induced subgraph
 of high connectivity and high chromatic number. We give a new proof of 
 this theorem with a better bound. 
\end{abstract}

\section{Introduction}\label{sec:intro}

All graphs in this paper are finite, simple, and non-null (unless
specified otherwise). A \emph{cutset} of a graph $G$ is a (possibly empty)
subset $C$ of $V(G)$ such that $G \sm C$ is disconnected. A graph is
\emph{$k$-connected} if it contains at least $k+1$ vertices and does
not contain a cutset of size at most $k-1$. We denote by $\chi(G)$ the
chromatic number of $G$. The starting point of this research is the
following theorem.

\begin{theorem}[Alon, Kleitman, Saks, Seymour, and Thomassen~\cite{alon87}]
 \label{alon}
 Let $k$ and $c$ be positive integers, and let $G$ be a graph such
 that $\chi(G) > \max\{c+10k^2+1,100k^3\}$. Then $G$ contains a
 $(k+1)$-connected induced subgraph of chromatic number greater than $c$.
\end{theorem} 

This theorem was improved by Chudnovsky, Penev, Scott, and
Trotignon~\cite{CPST:subst} who showed that the condition $\chi(G) >
\max\{c+2k^2,2k^2+k\}$ is sufficient. The proof
from~\cite{CPST:subst} relies on an ad hoc induction hypothesis,
which roughly states that upper bounds on the chromatic number are
preserved under gluing along a fixed number of vertices.

Here we improve the bound further (we show that the condition $\chi(G) 
> \max\{c+2k-2,2k^2\}$ is sufficient, see Theorem~\ref{th:col}). Our 
proof relies on structural properties. We prove that in any graph
$G$ of high minimum degree, there is a highly connected induced
subgraph $H$ such that only a small number of vertices of $H$ have
neighbors outside of $H$ (see Theorem~\ref{th:sf}). This generalizes a 
classical result on connectivity due to Mader~\cite{mader:4k}. 

There are several variants of the structural result. One variant involves 
classes of graphs built by repeatedly gluing prescribed basic blocks along 
a fixed number $k$ of vertices. We show that these graphs admit a 
``small'' cutset (its size is bounded by a function that depends only on 
$k$) that isolates a basic block (see Theorems~\ref{th:class} and~\ref{th:phi}). 
We provide several bounds for the functions that appear in our results. 

Our results are stated precisely in Section~\ref{sec:main}, and they are proven 
in Sections~\ref{sec:thPart}-\ref{sec:col}.

\section{Main results}
\label{sec:main}

We denote by $\delta(G)$ the minimum degree of a vertex of a graph
$G$, and we denote by $d(G)$ the average degree of $G$. If $G$ is a
graph and $S \subseteq V(G)$, we denote by $N_G(S)$ the 
\emph{neighborhood} of $S$, that is, the set of all vertices in $V(G)
\smallsetminus S$ that have a neighbor in $S$; we denote by $N_G[S]$ 
the {\em closed neighborhood} of $S$, that is, the set of all vertices 
of $G$ that either belong to $S$ or have a neighbor in $S$ (thus, 
$N_G[S] = S \cup N_G(S)$); and we denote by $\partial_G(S)$ the 
\emph{frontier} of $S$, that is, the set of all vertices in $S$ that 
have a neighbor in $V(G) \smallsetminus S$, (thus, $\partial_G(S) = 
N_G(V(G) \sm S)$). If $H$ is an induced subgraph of $G$, we sometimes 
write $N_G(H)$, $N_G[H]$, and $\partial_G(H)$ instead of $N_G(V(H))$, 
$N_G[V(H)]$, and $\partial_G(V(H))$, respectively. For a vertex $v$ 
of $G$, we sometimes write $N_G(v)$ and $N_G[v]$ instead of $N_G(\{v\})$ 
and $N_G[\{v\}]$, respectively. Furthermore, when clear from the context, 
we omit subscripts and write simply $N$ and $\partial$ instead of $N_G$ 
and $\partial_G$, respectively. Our main result is the following theorem. 

\begin{theorem} 
 \label{th:sf} 
 Let $k$ be a positive integer, and let $G$ be a graph. If
 $\delta(G) > 2k^2-1$, then $G$ contains a $(k+1)$-connected induced
 subgraph $H$ such that $\partial(H) \subsetneq V(H)$ and
 $|\partial(H)| \leq 2k^2-1$.
\end{theorem}

Theorem \ref{th:sf} can be thought of as a generalization of the
following theorem of Mader.

\begin{theorem}[Mader~\cite{mader:4k}] 
 \label{th:maderdelta} 
 Let $k$ be a positive integer, and let $G$ be a graph. If
 $\delta(G) \geq 4k$, then $G$ contains a $(k+1)$-connected induced
 subgraph.
\end{theorem}

On the one hand, Theorem~\ref{th:sf} gives a stronger result because of
the statement about $\partial(H)$, but on the other hand,
Theorem~\ref{th:maderdelta} is stronger, because the assumption on
$\delta$ is weaker. We remark that Theorem~\ref{th:maderdelta} is
usually given in terms of average (rather than minimum) degree, as
follows.

\begin{theorem}[Mader \cite{mader:4k}] 
 \label{th:maderd}
 Let $k$ be a positive integer, and let $G$ be a graph. If
 $d(G) \geq 4k$, then $G$ contains a $(k+1)$-connected induced
 subgraph.
\end{theorem}

Obviously, Theorem~\ref{th:maderdelta} is a corollary of
Theorem~\ref{th:maderd}, because for any graph $G$, $d(G) \geq
\delta(G)$. One may wonder whether it might be possible to strengthen
Theorem~\ref{th:sf} by replacing $\delta$ with $d$. Let us see that
this is not possible. Consider an integer $d\geq 3$ and a
$(2d-2)$-regular graph $G_0$ of girth $d$ (such a graph exists and 
is called a \emph{$(2d-2, d)$-cage}, see for instance~\cite{ExJa:cage}). 
Now, let $G$ be the graph obtained by adding a pendant vertex at each 
vertex of $G_0$. The average degree of $G$ and the girth of $G$ are 
both equal to $d$. A $2$-connected induced subgraph $H$ of $G$ must contain 
a cycle, so because of the girth of $G$, it must contain at least $d$ 
vertices. Moreover, $H$ cannot contain any of the pendant vertices, and 
so $|\partial(H)| \geq d$. Thus, $G$ does not contain a $2$-connected 
induced subgraph with a ``small'' frontier. We state this formally below.

\begin{theorem} 
 \label{th:d} 
 For every integer $d$, there exits a graph $G$ such that
 $d(G) \geq d$ and such that every $2$-connected induced subgraph $H$
 of $G$ satisfies $|\partial(H)| \geq d$.
\end{theorem}

\subsection*{Cut-partitions}

We now turn to a restatement of Theorem~\ref{th:sf} that is convenient
for the rest of the paper. We begin with some definitions. Given a
graph $G$, a set $S \subseteq V(G)$, and a vertex $v \in V(G)
\smallsetminus S$, we say that $v$ is {\em complete} (respectively:
{\em anti-complete}) to $S$ in $G$ provided that $v$ is adjacent
(respectively: non-adjacent) to every vertex of $S$. Given disjoint
subsets $X$ and $Y$ of $V(G)$, we say that $X$ is {\em complete}
(respectively: {\em anti-complete}) to $Y$ in $G$ provided that every
vertex of $X$ is complete (respectively: anti-complete) to $Y$ in
$G$. A \emph{cut-partition} of a graph $G$ is a partition $(A,B,C)$ of
$V(G)$ such that $A$ and $B$ are non-empty ($C$ may possibly be
empty), and $A$ is anti-complete to $B$ in $G$. Clearly, if $(A,B,C)$
is a cut-partition of $G$, then $C$ is a cutset of $G$. Conversely,
every cutset of $G$ gives rise to at least one cut-partition of
$G$. Furthermore, note that if $(A,B,C)$ is a cut-partition of $G$, 
then $N_G(A) \subseteq C$, $N_G[A] \subseteq A \cup C$, and 
$\partial_G(A \cup C) \subseteq C$. 

Theorem \ref{th:sf} can be restated in terms of cut-partitions,
as follows.

\begin{theorem} 
 \label{th:conn} 
 Let $k$ be a positive integer, and let $G$ be a graph. Then at least
 one of the following holds:
\begin{itemize} 
\item $G$ is $(k+1)$-connected; 
\item $G$ admits a cut-partition $(A,B,C)$ such that $G[A \cup C]$ is
 $(k+1)$-connected and $|C| \leq 2k^2-1$;
\item $G$ contains a vertex of degree at most $2k^2-1$.
\end{itemize} 
\end{theorem} 

Let us check that Theorems~\ref{th:sf} and~\ref{th:conn} are indeed
equivalent. Fix a graph $G$; we may assume that $\delta(G) > 2k^2-1$,
for otherwise $G$ satisfies both statements. Suppose first that $G$
satisfies Theorem~\ref{th:sf}, and let $H$ be as in that theorem. If
$H = G$, then $G$ satisfies the first outcome of Theorem~\ref{th:conn},
and otherwise, we set $A = V(H) \smallsetminus \partial(H)$, $B = V(G)
\smallsetminus V(H)$, and $C = \partial(H)$, and we observe that
$(A,B,C)$ satisfies the second outcome of
Theorem~\ref{th:conn}. Conversely, suppose that $G$ satisfies
Theorem~\ref{th:conn}. Then since $\delta(G) > 2k^2-1$, $G$ satisfies
one of the first two outcomes of Theorem~\ref{th:conn}. If $G$
satisfies the first outcome of Theorem~\ref{th:conn}, then we set $H =
G$, and otherwise, we let $(A,B,C)$ be as in the second outcome of
Theorem~\ref{th:conn}, and we set $H = G[A \cup C]$. This shows that
Theorems~\ref{th:sf} and \ref{th:conn} are indeed equivalent. We prove
Theorem~\ref{th:conn} (and therefore Theorem~\ref{th:sf}) in
Section~\ref{sec:thPart}, where we also provide a polynomial time
algorithm that actually finds the objects whose existence is guaranteed 
by Theorem~\ref{th:conn}.

\subsection*{Extreme decomposition}

We now turn to another motivation for this research. A class of
graphs is \emph{hereditary} if it is closed under taking induced
subgraphs and isomorphism. Many interesting theorems have the
following general form: every graph of a given heredatary class can be
obtained from certain ``basic'' graphs by repeatedly applying certain
operations. Here we are interested in the following operation. We say
that $G$ is obtained from $G_1$ and $G_2$ by \emph{gluing along at
 most $k$ vertices} if $G_1[V(G_1) \cap V(G_2)] = G_2[V(G_1) \cap
V(G_2)]$, $|V(G_1) \cap V(G_2)| \leq k$, and $G = G_1 \cup G_2$. (Note
that $G_1$ and $G_2$ are both induced subgraphs of $G$. Furthermore,
if $V(G_1) \cap V(G_2) = \emptyset$, then $G$ is simply the disjoint
union of $G_1$ and $G_2$.)

If $k$ is a positive integer and $\cal G$ a class of graphs, we denote
by ${\cal G}^k$ the \emph{$k$-closure} of $\cal G$, that is, the
inclusion-wise smallest class that includes $\cal G$ and is
closed under the operation of gluing along at most $k$ vertices. We
sometimes refer to $\cal G$ as the \emph{basic class} and to its
members as \emph{basic graphs}. 

Note that if $\mathcal{G}$ is a hereditary class, then $\mathcal{G}^k$
is hereditary and closed under disjoint unions. Furthermore, it is
easy to see that if $\mathcal{G}$ is hereditary, then $\mathcal{G}^k$
is the inclusion-wise maximal hereditary class for which the statement
{\it every graph in the class is either in $\cal G$ or admits a cutset
 of size at most $k$} is true.

We frequently use the following simple lemma. 

\begin{lemma}
 \label{l:inG}
 If $k$ is a positive integer and $\cal G$ is a hereditary class of
 graphs, then every graph that belongs to $\mathcal{G}^k$ and does not
 admit a cutset of size at most $k$ belongs to $\mathcal{G}$. Thus, 
 every $(k+1)$-connected graph from ${\cal G}^k$ and
 every complete graph from ${\cal G}^k$ is in $\cal G$.
\end{lemma}

\begin{proof}
 The first statement is immediate from the definition of $\mathcal{G}^k$: 
 every graph in ${\cal G}^k$ is in $\cal G$ or has a cutset of size at most $k$. 
 Since $(k+1)$-connected graphs and complete graphs do not admit a cutset of size 
 at most $k$, the second statement follows from the first. 
\end{proof}

The following theorem is a direct consequence of Theorem~\ref{th:conn}
and Lemma~\ref{l:inG}.

\begin{theorem} 
 \label{th:class} 
 Let $k$ be a positive integer, let $\cal G$ be a hereditary class of
 graphs, and let $G$ be a graph in ${\cal G}^k$. Then at least one of
 the following holds:
\begin{itemize} 
\item $G \in \mathcal{G}$; 
\item $G$ admits a cut-partition $(A,B,C)$ such that $G[A \cup C] \in \mathcal{G}$ and $|C| \leq 2k^2-1$; 
\item $G$ has a vertex of degree at most $2k^2-1$. 
\end{itemize} 
\end{theorem}

One might wonder whether the third outcome from Theorem~\ref{th:class} is 
truly necessary. For $k=1$, it is easy to see that the third outcome is in fact 
unnecessary: for every hereditary class $\mathcal{G}$, and every graph $G \in 
\mathcal{G}^1$, we have that either $G \in \mathcal{G}$, or $G$ admits a 
cut-partition $(A,B,C)$ such that $G[A \cup C] \in \mathcal{G}$ and $|C| \leq 1$. 
However, for $k \geq 2$, the third outcome is indeed necessary: the statement 
obtained from Theorem~\ref{th:class} by removing the third outcome is false, even 
if we increase the bound of $2k^2-1$ in the second outcome. To see this, suppose 
that $k \geq 2$ and that $\mathcal{G}$ is the class of all complete graphs. Then 
the graph $G$ represented in Figure~\ref{fig:twodiamonds} belongs to 
$\mathcal{G}^k$, does not belong to $\mathcal{G}$, and (since it does not admit 
a clique cutset) does not admit a cut-partition $(A,B,C)$ such that 
$G[A \cup C] \in \mathcal{G}$. 

\begin{figure}
\begin{center}
\includegraphics{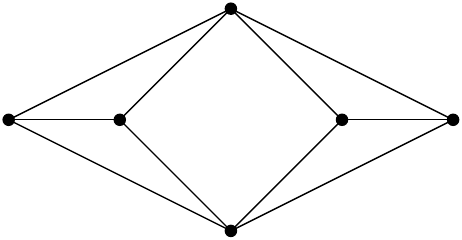}
\end{center} 
\caption{A graph that belongs to the $2$-closure of the class of all complete graphs, is not complete, and does not admit a cut-partition $(A,B,C)$ such that $A \cup C$ is a clique} \label{fig:twodiamonds} 
\end{figure}

\medskip 

Theorem~\ref{th:class} is an example of what is called an 
\emph{extreme decomposition} result. Such a result states that if 
a graph is decomposable by a certain kind of decomposition, then 
there is a decomposition (sometimes not exactly of the same kind; in 
Theorem~\ref{th:class}, for instance, we have to increase the size 
of the cutset from $k$ to $2k^2-1$) such that one of the blocks 
is basic (in Theorem~\ref{th:class}, the block $G[A\cup C]$). This 
can be very useful in various contexts: for proofs by induction, 
a basic block can be easier to handle, and for algorithms, 
recursing in a basic block can be faster than in a graph that 
needs to be decomposed further. Here is an immediate corollary of
Theorem~\ref{th:class}.

\begin{theorem} 
 \label{th:phi} Let $k$ be a positive integer, let $\cal G$ be a
 hereditary class of graphs, and let $G$ be a graph in ${\cal
 G}^k$. Then either $G \in \cal G$, or $G$ admits a cut-partition
 $(A, B, C)$ such $G[A] \in \cal G$ and $|C| \leq 2k^2-1$.
\end{theorem} 

\begin{proof} 
 If $G$ satisfies the first outcome of Theorem \ref{th:class}, then
 we are done. If it satisfies the second outcome of
 Theorem~\ref{th:class}, the result follows from the fact that
 $\mathcal{G}^k$ is hereditary. Finally, suppose that $G$ satisfies
 the third outcome of the theorem. Fix $v \in V(G)$ such that ${\rm
 deg}(v) = \delta(G) \leq 2k^2-1$. If $v$ has a non-neighbor in
 $G$, then we set $A = \{v\}$, $B = V(G) \smallsetminus N[v]$, 
 and $C = N(v)$, and we are done. So suppose that $v$ does
 not have any non-neighbors in $G$. Then the fact that ${\rm deg}(v)
 = \delta(G)$ implies that $G$ is a complete graph, and so by
 Lemma~\ref{l:inG}, $G \in \mathcal{G}$. This completes the argument.
\end{proof}

In Theorem~\ref{th:phi}, we no longer insist that $G[A\cup C]$ be
basic, but merely that $G[A]$ be basic (so the outcome about the
degree is no longer needed, since a one-vertex graph is basic). In some
applications, it suffices to have $G[A]$ basic.

\medskip{}

Let us now observe that Mader's theorems (Theorems~\ref{th:maderdelta}
and~\ref{th:maderd}) can be stated in terms of basic classes and
$k$-closure. To see this, note that for any graph $G$, the statement ``$G$
has no $(k+1)$-connected induced subgraph'' is equivalent to ``$G$ is
in ${\cal G}^k$,'' where the basic class $\cal G$ is the class of all graphs
on at most $k+1$ vertices. Therefore, the following is simply the
contrapositive statement of Theorem~\ref{th:maderd}. 

\begin{theorem}[Mader~\cite{mader:4k}]
 \label{mader:c}
 Let $k$ be a positive integer. If $\cal G$ is the class of all
 graphs on at most $k+1$ vertices, then every graph $G$ in ${\cal
 G}^k$ has average degree smaller than $4k$, and consequently
 contains a vertex of degree at most $4k-1$.
\end{theorem}

Let $k$, $\mathcal{G}$, and $G$ be as in Theorem~\ref{mader:c}; the theorem 
then guarantees that $\delta(G) \leq 4k-1$. Let $a$ be a vertex of degree 
$\delta(G)$. If $a$ has no non-neighbors, then the fact that ${\rm deg}(a) = 
\delta(G)$ implies that $G$ is a complete graph, and so by Lemma~\ref{l:inG}, 
$G \in \mathcal{G}$. Otherwise, we set $A = \{a\}$, $B = V(G) \smallsetminus 
N[a]$, and $C = N(a)$, and we observe that $(A,B,C)$ is a cut-partition of 
$G$ such that $G[A] \in \mathcal{G}$ and $|C| \leq 4k-1$. This shows that 
Theorem~\ref{th:maderdelta} is, in a sense, a special case of 
Theorem~\ref{th:phi} (however, Theorem~\ref{th:phi} yields a weaker upper 
bound than Mader's Theorem~\ref{th:maderdelta} does; this is because in 
Theorem~\ref{th:phi}, the basic class $\mathcal{G}$ is an arbitrary hereditary 
class, whereas in Theorem~\ref{mader:c}, this is not the case).

\medskip{}

It would be natural to have a theorem of the following form: there
exists a function $f$ such that for every positive integer $k$ and
every graph $G$ such that $\delta(G) \geq f(k)$, there exists a
$(k+1)$-connected induced subgraph $H$ of $G$ such that $|N(H)| \leq
f(k)$. This statement would be very similar to Theorem~\ref{th:sf}
($\partial(H)$ is replaced by $N(H)$), and it would be the ``connectivity 
version'' of Theorem~\ref{th:phi}, just as Theorem~\ref{th:sf} is the 
``connectivity version'' of Theorem~\ref{th:class}. Unfortunately, the 
statement is false even for $k = 1$. 

We fix a positive integer $k$ and build a counterexample inductively. 
We define $G_0$ to be the graph on one vertex, and for every integer 
$c > 0$, we build $G_c$ by taking two disjoint copies $G'$ and $G''$
of $G_{c-1}$, and adding a new vertex $v$ complete to them. The
key property of $G_c$ is that any induced subgraph $H$ of $G_c$ with 
no cutset of size at most $k$ satisfies $|N(H)| \geq c$. (In particular, 
any $(k+1)$-connected induced subgraph $H$ of $G_c$ satisfies $|N(H)| 
\geq c$, and furthermore, $\delta(G_c) \geq c$.) For $c=0$, this is 
obvious. For $c > 0$, note first that $V(H)$ cannot intersect both $V(G')$ 
and $V(G'')$, because $H$ would then have a cutset of size at most one 
(indeed, if $v \in V(H)$, then $v$ would be a cutvertex of $H$, and 
otherwise, $H$ would be disconnected), a contradiction to the assumption 
that $H$ has no cutset of size at most $k$. So, up to symmetry, $V(H) 
\subseteq V(G') \cup \{v\}$. If $v \in V(H)$, then $N_{G_c}(H)$ contains 
all vertices of $G''$, and so $|N_{G_c}(H)| \geq |V(G_{c-1})| \geq 
\delta(G_{c-1})+1 \geq c$. Otherwise, $v \notin V(H)$, and so 
$N_{G_c}(H)$ contains all vertices of $N_{G'}(H)$ (by the induction 
hypothesis, there are at least $c-1$ of them), plus $v$. In either case, 
$N_{G_c}(H)$ contains at least $c$ vertices. What we just proved is 
stated formally below. 

\begin{theorem} 
 \label{no-phi-c} 
 Let $k$ be a positive integer. Then for every integer
 $c$, there exists a graph $G$ such that every 
 induced subgraph $H$ of $G$ that has no cutset of size at most $k$
 satisfies $|N(H)| \geq c$.
\end{theorem}

\subsection*{Bounds}

We now study how far we can improve the bounds in the theorems
mentioned thus far. We call $\psi_c$, $\psi$, and $\varphi$ the best
possible bounds in Theorems~\ref{th:conn}, \ref{th:class}, 
and~\ref{th:phi} respectively. More precisely:
\begin{itemize}
\item $\psi_c$ ($c$ stands for ``connectivity'') is the smallest
 function such that for all positive integers $k$ and all 
 graphs $G$, either $G$ is $(k+1)$-connected, or $G$ admits a
 cut-partition $(A,B,C)$ such that $G[A \cup C]$ is $(k+1)$-connected
 and $|C| \leq \psi_c(k)$, or $G$ contains a vertex of degree at most
 $\psi_c(k)$.
\item $\psi$ is the smallest function such that for all positive
 integers $k$, all hereditary classes $\cal G$, and all graphs $G \in
 {\cal G}^k$, either $G \in \cal G$, or $G$ admits a cut-partition
 $(A, B, C)$ such that $G[A \cup C] \in \cal G$ and $|C| \leq
 \psi(k)$, or $G$ has a vertex of degree at most $\psi(k)$.
\item $\varphi$ is the smallest function such that for all positive
 integers $k$, all hereditary classes $\cal G$, and all graphs $G \in
 {\cal G}^k$, either $G \in \cal G$, or $G$ admits a cut-partition $(A,
 B, C)$ such that $G[A] \in \cal G$ and $|C| \leq \varphi(k)$.
\end{itemize} 
The existence of functions $\psi_c$, $\psi$, and $\varphi$ follows
from Theorems~\ref{th:conn}, \ref{th:class}, and \ref{th:phi},
respectively. We remark that all three of these functions are 
non-decreasing. Indeed, the fact that $\psi_c$ is non-decreasing 
follows form the fact that every $(k+2)$-connected graph is also 
$(k+1)$-connected, and the fact that $\psi$ and $\varphi$ are 
non-decreasing follows from the fact that for all hereditary 
classes $\mathcal{G}$, we have that $\mathcal{G}^k \subseteq 
\mathcal{G}^{k+1}$. 

Unfortunately, we have not been able to find exact formulas for the 
functions $\psi_c$, $\psi$, and $\varphi$. We have, however, been able 
to compute certain upper and lower bounds for these three functions, 
as well as exact values for $k = 2$, as stated in the following theorem. 

\begin{theorem} \label{th:bounds} For all positive integers $k$, all the following hold: 
\begin{itemize} 
\item[(1)] $2k-1 \leq \varphi(k) \leq \psi(k) \leq \psi_c(k) \leq 2k^2-1$; 
\item[(2)] $k^2+k-1 \leq \psi(k) \leq \psi_c(k) \leq 2k^2-1$; 
\item[(3)] $\frac{1}{4}k\log_2k < \varphi(k) \leq 2k^2-1$. 
\end{itemize}
Furthermore, 
\begin{itemize} 
\item[(4)] $\varphi(2) = \psi(2) = \psi_c(2) = 5$. 
\end{itemize} 
\end{theorem} 

We prove Theorem~\ref{th:bounds} in Section~\ref{sec:bounds}. Note 
that part (2) of this theorem implies that functions 
$\psi$ and $\psi_c$ are quadratic. However, we have not been able to 
determine the order of the function $\varphi$: part (3) gives a 
lower bound of order $k\log k$ and an upper bound of order $k^2$ for 
$\varphi(k)$. We observe that for small values of $k$, part (1) gives 
a better lower bound for $\varphi(k)$ than part (3) does, but for 
large values of $k$, the lower bound from (3) is better. 

We remark that the fact that $\psi_c(k) \leq 2k^2-1$ follows from 
Theorem~\ref{th:conn}, and it is an easy exercise to establish 
the inequalities $\varphi(k) \leq \psi(k) \leq \psi_c(k)$. The inequalities 
$2k-1 \leq \varphi(k)$ and $k^2+k-1 \leq \psi(k)$ are obtained by considering a 
particular graph from the $k$-closure of the class of all complete graphs (the 
same graph yields both of these inequalities). We obtain the inequality 
$\frac{1}{4}k\log_2k < \varphi(k)$ by constructing another particular graph from 
the $k$-closure of the class of all complete graphs. Note that part (1) 
of Theorem~\ref{th:bounds} implies that $\varphi(1) = \psi(1) = \psi_c(1) = 1$ 
(this is also easy to prove from scratch, as the reader can check). As 
part (4) states, we have also been able to deal with the case $k = 2$. To prove 
part (4), it suffices to prove inequalities $5 \leq \varphi(2)$ and $\psi_c(2) 
\leq 5$, for the rest then follows from part (1). We obtain the inequality 
$5 \leq \varphi(2)$ by constructing a suitable graph from the $2$-closure of a 
certain hereditary class. The proof of the inequality $\psi_c(2) \leq 
5$ is more involved, and we refer the reader to Section~\ref{sec:bounds}.

Note that the bound in Mader's theorem is linear, while our general
lower bound for $\varphi$ (the bound from part (3) of Theorem~\ref{th:bounds}) 
is not. This is because in Mader's theorem (Theorem~\ref{mader:c}), the basic 
class $\cal G$ is not arbitrary. This shows how $\varphi$ can be sensitive to 
$\cal G$.

\subsection*{Coloring}

We now give the application of our results mentioned in the Introduction. 

\begin{theorem} \label{th:col} Let $k$ be a positive and $c$ a
 non-negative integer, and let $G$ be a graph such that $\chi(G) >
 \max\{c+2k-2,2k^2\}$. Then $G$ contains a $(k+1)$-connected induced
 subgraph of chromatic number greater than $c$.
\end{theorem}
\noindent 

Clearly, Theorem~\ref{th:col} is an improvement of Theorem~\ref{alon}. 
To see how our method works, we prove the following proposition, which 
is a weaker version of Theorem~\ref{th:col}. 
\begin{proposition} \label{prop:coloring} Let $k$ and $c$ be positive integers, 
and let $G$ be a graph such that $\chi(G) > c+2k^2-1$. Then $G$ contains a 
$(k+1)$-connected induced subgraph $H$ of chromatic number greater 
than $c$. 
\end{proposition} 
\begin{proof} 
We assume inductively that the statement holds for graphs on fewer than 
$|V(G)|$ vertices. We now apply Theorem~\ref{th:conn} to $G$. If
$G$ is $(k+1)$-connected, then we set $H = G$, and we are done. If $G$
contains a vertex $v$ of degree at most $2k^2 - 1$, then we see that
$\chi(G \sm v) > c+2k^2-1$ (we use the fact that $c > 0$), and we apply the 
induction hypothesis to $G\sm v$. We may therefore assume that $G$ admits 
a cut-partition $(A,B,C)$ such that $G[A \cup C]$ is $(k+1)$-connected and 
$|C| \leq 2k^2-1$. We may further assume that $\chi(G[B \cup C]) \leq 
c+2k^2-1$, for otherwise, we apply the induction hypothesis to 
$G[B \cup C]$, and we are done. If $\chi(G[A]) \leq c$, then we color 
$G[B \cup C]$ with at most $c+2k^2-1$ colors, at most $|C| \leq 2k^2-1$ of 
which are used on $C$, and we color $G[A]$ with the remaining $c$ colors, 
thus obtaining a proper coloring of $G$ that uses only $c+2k^2-1$ colors, 
contrary to the assumption that $\chi(G) > c+2k^2-1$. Thus, $\chi(G[A]) > c$, 
and consequently, $\chi(G[A \cup C]) > c$. We now set $H = G[A \cup C]$, 
and we are done. 
\end{proof} 

The proof of Theorem \ref{th:col} is given in Section~\ref{sec:col}. 
It is similar to (but more complicated than) the proof of 
Proposition~\ref{prop:coloring}. In fact, we could not derive 
Theorem~\ref{th:col} from theorems stated earlier in this section 
(which do not seem to lead to anything stronger than 
Proposition~\ref{prop:coloring}). Rather, we use a more technical 
result (namely, Corollary~\ref{cor-main}) proven in 
Section~\ref{sec:thPart}. 

\medskip 

We now show that Theorem~\ref{th:col} has a corollary stated in terms
of operations that perverve $\chi$-boundedness, a notion introduced by
Gy\'arf\'as in~\cite{gyarfas:perfect}. When $G$ is a graph, we denote
by $\omega(G)$ the size of a maximum clique of $G$. A hereditary
class $\mathcal{G}$ is {\em $\chi$-bounded} if there exists a function
$f:\mathbb{N} \rightarrow \mathbb{N}$ such that for all graph $G \in
\mathcal{G}$, $\chi(G) \leq f(\omega(G))$. Under these circumstances,
we also say that $\mathcal{G}$ is $\chi$-bounded {\em by} $f$, as well
as that $f$ is a {\em $\chi$-bounding function} of $\mathcal{G}$. It is
easy to see that if a hereditary class is $\chi$-bounded, then there 
exists a non-decreasing function $f$ such that the class is 
$\chi$-bounded by $f$.

\begin{theorem}\label{chi-new} 
 Let $k$ be a positive integer, and let $\mathcal{G}$ be a hereditary
 class of graphs, $\chi$-bounded by a non-decreasing
 function~$f$. Then $\mathcal{G}^k$ is $\chi$-bounded by the function
 defined by $g(n) = \max\{f(n)+2k-2, 2k^2\}$.
\end{theorem} 

\begin{proof}
Suppose for a contradiction that a graph $G\in {\cal G}^k$ satisfies $\chi(G) >
\max\{f(\omega(G))+2k-2, 2k^2\}$. By Theorem~\ref{th:col}, there
exits a $(k+1)$-connected induced subgraph $H$ of $G$ with chromatic
number greater than $f(\omega(G))$. Since $H$ is $(k+1)$-connected
and is in ${\cal G}^k$, Lemma~\ref{l:inG} implies that $H \in {\cal G}$. Since $f$
is non-deacreasing, we have $\chi(H) > f(\omega(G)) \geq f(\omega(H)),$ a 
contradiction to the fact that $\cal G$ is $\chi$-bounded by $f$ and $H \in 
\mathcal{G}$. 
\end{proof}

In Section~\ref{sec:col}, we prove that there is in a certain sense an 
equivalence between results about highly connected induced subgraphs 
of high chromatic number in graphs of high chromatic number, and results 
about preserving $\chi$-boundedness under the operation of gluing along 
a bounded number of vertices (see Proposition~\ref{chi-conn-class-equiv}).

\subsection*{Open questions}

Let us now mention a few open questions that arise naturally from the
results of this paper. While Theorem~\ref{th:bounds} answers certain 
questions about the functions $\varphi$, $\psi$, and $\psi_c$, it leaves 
a number of other questions open. While we could deal with the cases $k = 1$ and 
$k = 2$ (the former follows from part (1) and the latter from part (4) of 
Theorem~\ref{th:bounds}), we could not compute the exact values of $\varphi(k)$, 
$\psi(k)$, and $\psi_c(k)$ for $k \geq 3$. In fact, we do not even know whether 
$\varphi$, $\psi$, and $\psi_c$ are computable. Next, while we could prove that 
$\psi$ and $\psi_c$ are quadratic functions, we have not been able to determine 
the order of the function $\varphi$. We also do not know whether $\psi = \psi_c$. 
Further, even though $\varphi(k) = \psi(k)$ for $k \in \{1,2\}$, we believe (but 
have so far not been able to prove) that $\varphi$ and $\psi$ are different 
functions: we believe that there exists some constant $k_0 \geq 3$ such that for 
all $k \geq k_0$, $\varphi(k) < \psi(k)$. One reason for this is that our work 
on lower bounds for $\varphi$ and $\psi$ suggests that these two functions behave 
differently. In particular, the construction that gave us a quadratic lower bound 
for $\psi$ yields only a linear lower bound for $\varphi$. While we could 
ultimately prove that $\varphi$ is superlinear (with a lower bound of order 
$k\log k$), the construction that we needed in order to obtain this lower bound 
for $\varphi$ is a lot more complicated than the construction that gave us a 
quadratic lower bound for $\psi$.

Our work suggests that assuming a high minimum degree and assuming a
high average degree have different implications. Indeed, we could
generalize the minimum-degree version of Mader's theorem
(see Theorem~\ref{th:sf}), but we proved that the average-degree version
cannot be generalized in the same way (see Theorem~\ref{th:d}).
Therefore, we wonder whether the best bound is the same for the two
versions of Mader's theorem (Theorems~\ref{th:maderdelta} and~\ref{th:maderd}). 
All known proofs of these theorems rely on the average degree, including 
the proof due to Hajnal~\cite{hajnal:ext}, which established the best bound 
known so far. To support the idea that a direct proof and a different bound 
might exist for Theorem~\ref{th:maderdelta}, we give a proof of the following 
known special case. This proof is similar to our proof that $\psi_c(2)=5$ (the 
details are simpler and we obtain a slightly better value, namely $4$). The graph 
represented in Figure~\ref{fig:mindeg4} shows that Theorem~\ref{th:24} is best 
possible. To our knowledge, this proof is new, and it really relies on minimum 
degree rather than average degree. 

\begin{figure}
\begin{center}
\includegraphics{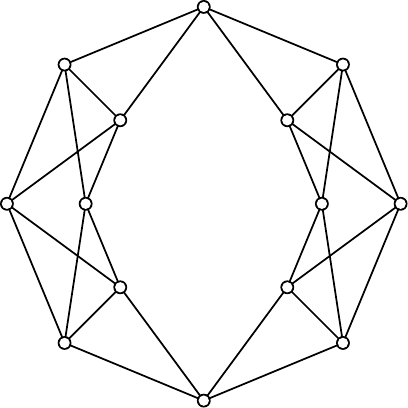}
\end{center}
\caption{A graph with minimum degree 4 and no 3-connected subgraph\label{fig:mindeg4}}
\end{figure}

\begin{theorem}[Mader~\cite{mader:4k}]
 \label{th:24}
 If $G$ has minimum degree greater than~4, then $G$ contains a
 3-connected induced subgraph. 
\end{theorem}

\begin{proof}
 We prove by induction on the number of vertices a variant of the contrapositive
 statement (which clearly implies the theorem):
 
\medskip

\noindent{\it For all graphs $G$, if $G$ has at least three vertices
 and $G$ contains no 3-connected induced subgraph, then $G$ contains either:
 \begin{enumerate}[(i)]
 \item\label{i:a} three vertices of degree at most $2$;
 \item\label{i:b} two vertices of degree at most $2$ and two vertices of degree at most $4$;
 \item\label{i:c} one vertex of degree at most $2$ and four vertices of degree at most $4$;
 \item\label{i:d} six vertices of degree at most $4$;
 \end{enumerate}}

If $|V(G)| = 3$, then (\ref{i:a}) holds, and so from here on, we assume that $|V(G)| 
\geq 4$.

Suppose first that $G$ contains a vertex $v$ of degree at most~$2$. We apply
the induction hypothesis to $G\sm v$. If $G\sm v$ satisfies~(\ref{i:a}), then $G$
satisfies~(\ref{i:a}) or~(\ref{i:b}). If $G\sm v$ satisfies~(\ref{i:b}), then $G$ 
satisfies~(\ref{i:a}), (\ref{i:b}), or~(\ref{i:c}). If $G\sm v$ 
satisfies~(\ref{i:c}), then $G$ satisfies~(\ref{i:b}) or~(\ref{i:c}). If $G\sm v$
satisfies~(\ref{i:d}), then $G$ satisfies~(\ref{i:c}). So we may assume that 
$\delta(G) \geq 3$. 

Since $G$ itself is not $3$-connected and $|V(G)| \geq 4$, $G$ has a cut-partition 
$(A,B,C)$ with $|C|\leq 2$. We set $G_A = G[A \cup C]$ and $G_B = G[B \cup C]$. 
Since $\delta(G) \geq 3$, we have $|V(G_A)|, |V(G_B)| \geq 4$, 
so we may apply the induction hypothesis to $G_A$ and $G_B$. Note that
the outcome~(\ref{i:a}) may not hold for $G_A$ or $G_B$ because $G$
has no vertex of degree at most~$2$. Also, all vertices of degree~$2$ of
$G_A$ and $G_B$ are in $C$.
 
If one of $G_A$ or $G_B$ satisfies~(\ref{i:d}), then $G$ satisfies~(\ref{i:d}). 
So we are left with three cases. If $G_A$ and $G_B$ both satisfy~(\ref{i:c}) 
then $G$ satisfies~(\ref{i:d}). If one of $G_A$ and $G_B$ satisfies~(\ref{i:c}) 
and the other one satisfies~(\ref{i:b}), then $G$ satisfies~(\ref{i:d}) 
(one vertex of degree at most~$4$ in $G$ is in $C$). If $G_A$ and $G_B$ both 
satisfy~(\ref{i:b}), then $G$ satisfies~(\ref{i:d}) (two vertices of degree at 
most~$4$ in $G$ are in $C$).
\end{proof}

\section{Proof of Theorems~\ref{th:sf} and~\ref{th:conn}}
\label{sec:thPart}

Given a positive integer $k$, a graph $G$, a set $Z \subseteq V(G)$,
and a vertex $v \in V(G) \smallsetminus Z$, we say that $v$ is {\em
 $k$-weak} with respect to $Z$ if $v$ has at most $k$ neighbors in $Z$,
and we say that $v$ is {\em $k$-strong} with respect to $Z$ if $v$ has at
least $k+1$ neighbors in $Z$. The {\em $k$-weight} of $v$ with
respect to $Z$, denoted by $w_Z^k(v)$, is defined as follows:

\begin{itemize} 
\item if $v$ has no neighbors in $Z$, then $w_Z^k(v) = 1$;
\item if $v$ has at least one neighbor in $Z$, and $v$ is $k$-weak with
 respect to $Z$, then $w_Z^k(v)$ is the number of neighbors that $v$
 has in $Z$;
\item if $v$ is $k$-strong with respect to $Z$, then $w_Z^k(v) = k$. 
\end{itemize} 

\noindent 
Given disjoint sets $Y,Z \subseteq V(G)$, the {\em $k$-weight} of $Y$
with respect to $Z$, denoted by $w_Z^k(Y)$, is the sum of $k$-weights of 
the vertices of $Y$ with respect to $Z$ (if $Y = \emptyset$, then 
$w_Z^k(Y) = 0$). Clearly, $|Y| \leq w_Z^k(Y) \leq k|Y|$. Furthermore, if 
$Z \neq \emptyset$, then $w_Z^k(Y) \leq |Y||Z|$. (If $Z = \emptyset$, 
then $w_Z^k(Y) = |Y|$.) Note also that if $(A,B,C)$ is a cut-partition of 
a graph $G$ such that $|C| \leq k$, then $w_B^k(C) \leq k|C| \leq k^2 
\leq 2k^2-1$. 

When clear from the context, we sometimes write ``weak,'' ``strong,'' 
and ``weight'' instead of ``$k$-weak,'' ``$k$-strong,'' and ``$k$-weight,'' 
respectively. Similarly, when there is no risk of confusion, we often omit 
the superscript $k$ and write simply $w_Z(v)$ and $w_Z(Y)$ instead of 
$w_Z^k(v)$ and $w_Z^k(Y)$, respectively. 

Given two cut-partitions $(A,B,C)$ and $(A',B',C')$ of a graph $G$, we
say that $(A',B',C')$ is {\em better} than $(A,B,C)$ if $A' \cup C'
\subsetneq A \cup C$ (equivalently: $B \subsetneq B'$).

\begin{lemma} \label{lemma-alg} Let $k$ be a positive integer. 
 There exists an algorithm with the following properties:
\begin{itemize} 
\item Input: a graph $G$ such that $\delta(G) > 2k^2-1$, and a
 cut-partition $(A,B,C)$ of $G$ such that $w_B(C) \leq 2k^2-1$.
\item Output: either the true statement ``$G[A \cup C]$ is
 $(k+1)$-connected,'' or a cut-partition $(A',B',C')$ of
 $G$ that is better than $(A,B,C)$ and satisfies $w_{B'}(C') \leq 2k^2-1$.
\item Running time: $O(n^{k+2})$, where $n$ is the number of vertices of the input graph $G$. 
\end{itemize} 
\end{lemma} 

\begin{proof} 
 Let $G$ be a graph on $n$ vertices such that $\delta(G) > 2k^2-1$, and let $(A,B,C)$
 be a cut-partition of $G$ such that $w_B(C) \leq 2k^2-1$. By a classical
 connectivity test, it can be tested in time $O(n^{k+2})$ whether
 $G[A \cup C]$ is $(k+1)$-connected, and if so, the algorithm
 outputs the true statement ``$G[A \cup C]$ is $(k+1)$-connected''
 and stops. So we may assume that $G[A \cup C]$ is not
 $(k+1)$-connected. Since every vertex in $A$ has degree at least
 $2k^2$, and all neighbors of vertices of $A$ are in $A\cup C$, we
 know that $|A\cup C| \geq 2k^2+1 \geq k+2$. It follows that $G[A \cup C]$ has a
 cutset $S$ such that $|S| \leq k$. The cutset $S$ and the components
 of $G[A \cup C] \sm S$ can be found in time $O(n^{k+2})$. Let $S_A$ be the 
 vertex-set of a component of $G[A \cup C] \sm S$, and let $S_B = (A \cup C) 
 \smallsetminus (S \cup S_A)$. Then $(S_A,S_B,S)$ is a cut-partition of 
 $G[A \cup C]$ (see Figure~\ref{fig:cutpartini}). Computing $w_B(C \cap S_A)$ and 
 $w_B(C \cap S_B)$ takes at most  $O(n^2)$ time, and since $w_B(C) \leq 2k^2 -1$, 
 up to the symmetry between $S_A$ and $S_B$, we may assume that $w_B(C \cap S_A) 
 \leq k^2-1$. (In particular, $|C \cap S_A| \leq k^2-1$.) 

\begin{figure}
\begin{center}
\includegraphics[scale=0.6]{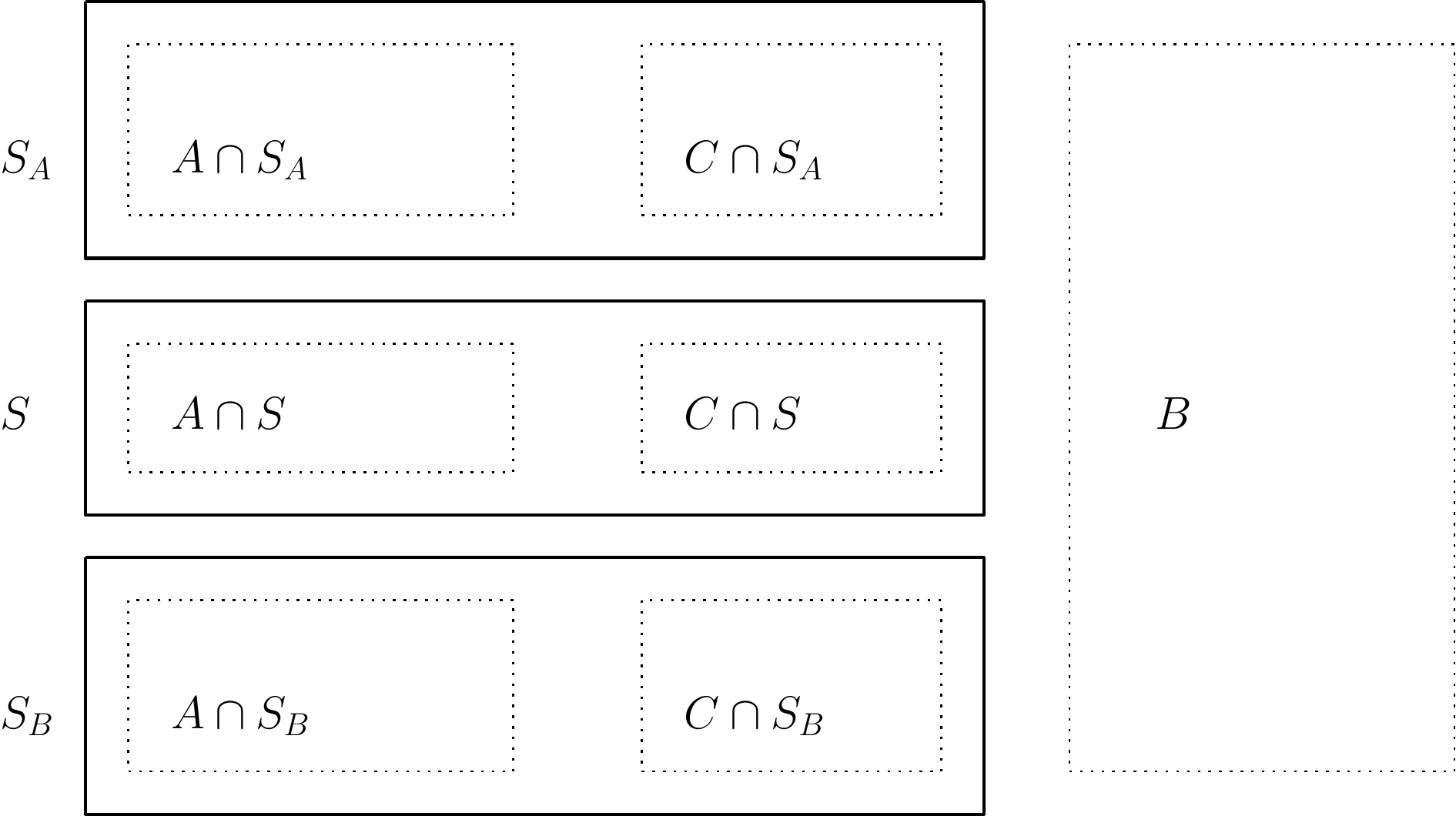}
\end{center} 
\caption{The cut-partition $(S_A,S_B,S)$ of $G[A \cup C]$} \label{fig:cutpartini} 
\end{figure}

 Suppose first that $A \cap S_A \neq \emptyset$. We then set $A' =
 A \cap S_A$, $B' = B \cup S_B$, and $C' = S \cup (C \cap S_A)$. We
 see that $(A', B', C')$ is a cut-partition of $G$ that is better
 than $(A, B, C)$. Note that $w_{B'}(C \cap S_A) = w_B(C \cap S_A)$
 because $S_A$ is anti-complete to $S_B$; consequently, $w_{B'}(C \cap S_A) \leq 
 k^2-1$ and $$w_{B'}(C')\, = \,w_{B'}(S) + w_{B'}(C \cap S_A)\, \leq \,k|S| + 
 (k^2 -1)\, \leq \, 2k^2-1.$$ The algorithm now outputs the cut-partition $(A',B',C')$ of 
 $G$ and stops. 

 From here on, we assume that $A \cap S_A = \emptyset$, so that $C \cap S_A
 \neq \emptyset$. It follows that for every vertex $v \in C \cap S_A$, 
 $N_G(v) \subseteq ((C \cap S_A) \sm \{v\}) \cup S \cup B$. In
 particular, if $v$ is weak with respect to $B$, then its degree is at
 most 
 \begin{displaymath}
 \begin{array}{rcl} 
 (|C \cap S_A|-1)+|S|+k & \leq & (k^2-1)-1+k+k 
 \\
 & = & k^2+2k-2
 \\
 & \leq & 2k^2-1,
 \end{array}
 \end{displaymath} 
 a contradiction to our assumption on $\delta(G)$.
 It follows that all vertices of $C \cap S_A$ are strong with
 respect to $B$. Since $C \cap S_A \neq \emptyset$ and $w_B(C) \leq 2k^2-1$, 
 this implies: 
\begin{eqnarray}
 w_B (C \cap S_A) & = &k|C \cap S_A| \label{e:1}\\ 
 |C|& \leq &2k^2-k \label{e:2}
\end{eqnarray}

\begin{figure}
\begin{center}
\includegraphics[scale=0.6]{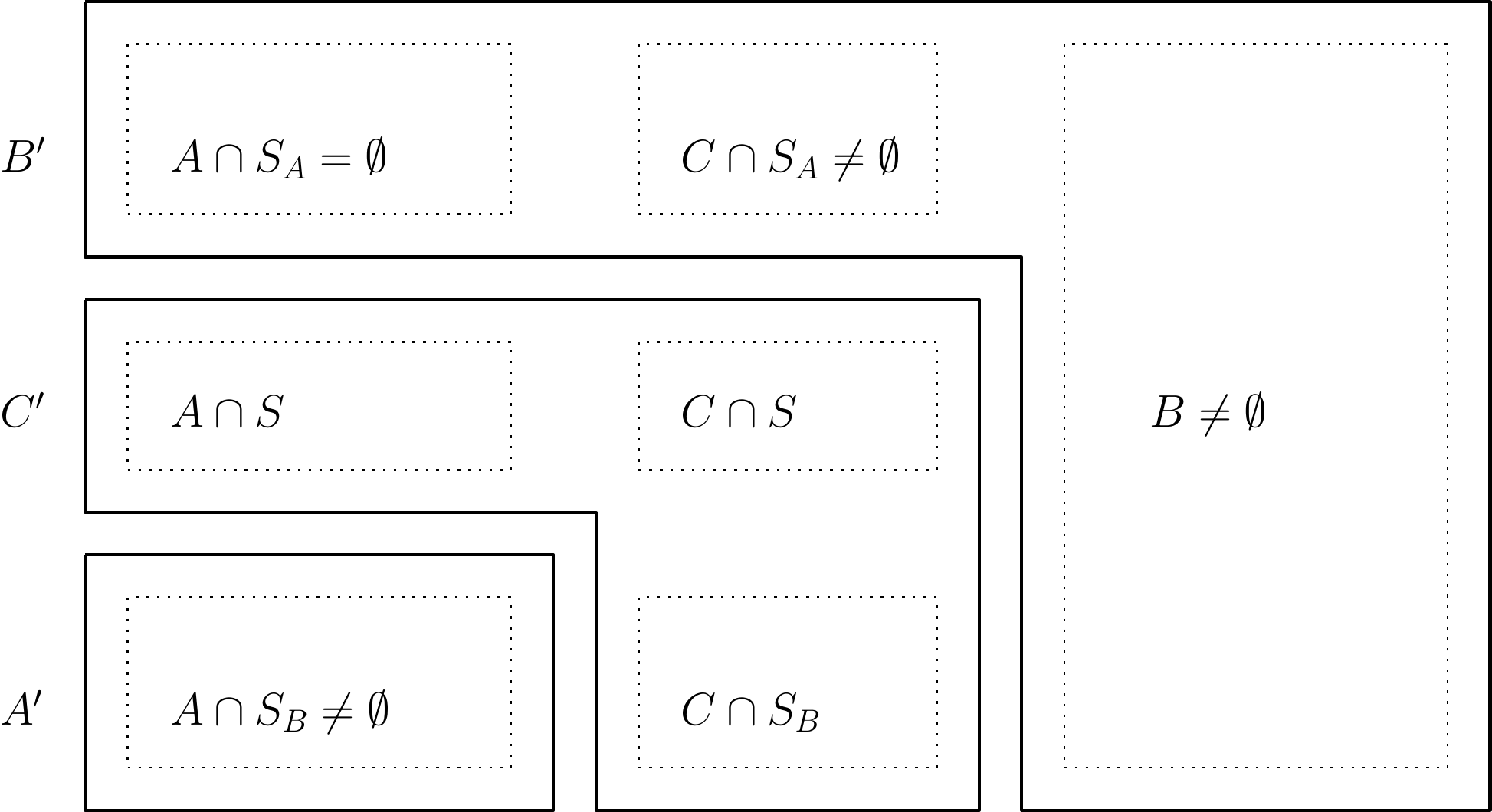}
\end{center} 
\caption{The cut-partition $(A',B',C')$ of $G$ for the case when $A \cap S_A = \emptyset$} \label{fig:cutpart} 
\end{figure}

If $A \cap S_B = \emptyset$, then by~(\ref{e:2}), 
a vertex $v$ from $A \subseteq S$ has at most
$$|S \sm \{v\}| + |C|\, \leq \, (k-1) + (2k^2 -k) \, = \, 2k^2-1$$
neighbors, a contradiction to our assumption on $\delta(G)$. Thus, 
$A \cap S_B \neq \emptyset$. We then set $A' = A \cap S_B$, 
$B' = B \cup S_A$, and $C' = S \cup (C \cap S_B)$ 
(see Figure~\ref{fig:cutpart}). We see that $(A',B',C')$ is a 
cut-partition of $G$ that is better than $(A,B,C)$, and furthermore, 
\begin{displaymath} 
\begin{array}{rclll} 
w_{B'}(C') & = & w_{B'}(S)+w_{B'}(C \cap S_B) & &
\\
& \leq & w_{C \cap S_A}(S)+w_B(C \cap S)+w_B(C \cap S_B) & & 
\\
& \leq & w_{C \cap S_A}(S)+w_B(C)-w_B(C \cap S_A) & &
\\
& \leq & |S||C \cap S_A|+w_B(C)-w_B(C \cap S_A) & & \text{since $C \cap S_A \neq \emptyset$} 
\\
& \leq & k|C \cap S_A|+w_B(C)-w_B(C \cap S_A) & & \text{since $|S| \leq k$} 
\\
& = & k|C \cap S_A|+w_B(C)-k |C\cap S_A| & & \text{by~(\ref{e:1})} 
\\
& = & w_B(C) & &
\\
& \leq & 2k^2-1. & &
\end{array} 
\end{displaymath} 
The algorithm now outputs $(A',B',C')$ and stops. 
\end{proof}

\begin{theorem} \label{theorem-alg} Let $k$ be a positive
 integer. There exists an algorithm whose input is a graph
 $G$, whose output is one of the following:
\begin{itemize} 
\item the true statement ``$G$ is $(k+1)$-connected''; 
\item a cut-partition $(A,B,C)$ of $G$ such that $G[A \cup C]$ is $(k+1)$-connected and $w_B(C) \leq 2k^2-1$; 
\item a vertex $v \in V(G)$ of degree at most $2k^2-1$; 
\end{itemize} 
\noindent 
and whose running time is $O(n^{k+3})$, where $n$ is the number of vertices of the input graph $G$. 
\end{theorem} 

\begin{proof} 
Here is an algorithm. 

\noindent\textbf{Step 1:} We first check in $O(n^2)$ time whether $G$ contains
a vertex of degree at most $2k^2-1$; if so, we stop, and the algorithm
returns one such vertex. From now on, we assume that $\delta(G) >
2k^2-1$. By examining all subsets of $V(G)$ of size at most $k$, we
determine in $O(n^{k+2})$ time whether $G$ has any cutsets of size at
most $k$. If $G$ has no such cutsets, then using the fact that $|V(G)|
\geq \delta(G)+1 \geq 2k^2+1 \geq k+2$, we determine that $G$ is
$(k+1)$-connected, and we are done. So assume that we found a cutset
$C$ of $G$ such that $|C| \leq k$. We now find the components of $G
\smallsetminus C$ in $O(n^2)$ time, we let $A$ be the vertex-set of
some component of $G \smallsetminus C$, and we set $B = V(G)
\smallsetminus (A \cup C)$. Clearly, $(A,B,C)$ is a cut-partition of
$G$, and furthermore, $w_B(C) \leq k|C| \leq k^2 \leq 2k^2-1$. We now go to Step 2.

\noindent\textbf{Step 2:} We call the algorithm from Lemma \ref{lemma-alg} with
input $G$ and $(A,B,C)$. If we obtain the statement that $G[A \cup C]$ is 
$(k+1)$-connected, then we stop and return the triple $(A,B,C)$. Otherwise, we 
obtain a cut-partition $(A',B',C')$ of $G$ that is better than $(A,B,C)$ and 
satisfies $w_{B'}(C') \leq 2k^2-1$. In this case, we set $(A,B,C) = 
(A',B',C')$, and we go back to Step 2.

By definition, if $(A,B,C)$ and $(A',B',C')$ are cut-partitions of $G$ such that 
$(A',B',C')$ is better than $(A,B,C)$, then $|A' \cup C'| < |A \cup C|$. This 
implies that we go through Step 2 at most $O(n)$ times (and in particular, the 
algorithm terminates). Since the running time of the algorithm from Lemma
\ref{lemma-alg} is $O(n^{k+2})$, we deduce that the running time of
the algorithm that we just described is $O(n^{k+3})$.
\end{proof} 

\begin{corollary} \label{cor-main} Let $k$ be a positive integer. Then
 for every graph $G$, at least one of the following holds:
\begin{itemize} 
\item $G$ is $(k+1)$-connected; 
\item $G$ admits a cut-partition $(A,B,C)$ such that $G[A \cup C]$ is $(k+1)$-connected and $w_B(C) \leq 2k^2-1$;
\item $G$ contains a vertex of degree at most $2k^2-1$. 
\end{itemize} 
\end{corollary} 
\begin{proof} 
This follows immediately from Theorem~\ref{theorem-alg}. 
\end{proof} 

Note that in the second outcome of Corollary \ref{cor-main}, we have
that $|C| \leq w_B(C) \leq 2k^2-1$. Thus, Theorem~\ref{th:conn} is an
immediate consequence of Corollary~\ref{cor-main}. As explained in
the Introduction, Theorem~\ref{th:sf} is equivalent to
Theorem~\ref{th:conn}. 

\medskip 

One may wonder whether the bound of $2k^2-1$ given in Corollary~\ref{cor-main} is best possible. While we are at this time not able to give a definitive answer to this question, we can show that, at least in the case when $k$ is a power of~$2$, the bound from Corollary~\ref{cor-main} is very close to being optimal (and for the case $k = 2$, the bound is indeed optimal). In particular, we have the following proposition.

\begin{proposition} \label{k-weight-power-of-2} Let $m$ and $d$ be positive integers, and let $k = 2^m$. Then there exists a graph $G$ such that $G$ is not $(k+1)$-connected, $\delta(G) \geq d$, and every cut-partition $(A,B,C)$ of $G$ such that $G[A \cup C]$ is $(k+1)$-connected satisfies $w_B^k(C) \geq 2k^2-\frac{k}{2}$. 
\end{proposition} 
\begin{proof} 
We may assume that $d \geq 3k$. Given $i \in \{0,\dots,m+1\}$, an {\em $i$-usable cover} of a graph $G$ is an ordered $2^i$-tuple $(A_1,\dots,A_{2^i})$ of cliques of $G$ such that all the following are satisfied: 
\begin{itemize} 
\item[(a)] $V(G) = \bigcup_{j=1}^{2^i} A_j$; 
\item[(b)] for all $j \in \{1,\dots,2^i\}$, $|A_j| = d+1$, $|\partial_G(A_j)| = 2k-2^{m-i+1}$ (and consequently, $|A_j \smallsetminus \partial_G(A_j)| > k+2^{m-i}$), and every vertex in $\partial_G(A_j)$ is $k$-strong with respect to $V(G) \smallsetminus A_j$; 
\item[(c)] the sets $A_1 \smallsetminus \partial_G(A_1),\dots,A_{2^i} \smallsetminus \partial_G(A_{2^i})$ are pairwise disjoint and anti-complete to each other; 
\item[(d)] for all non-empty sets $S \subseteq V(G)$ such that $G[S]$ is $(k+1)$-connected and $\partial_G(S) \subsetneq S$, there exists an index $j \in \{1,\dots,2^i\}$ such that $S = A_j$. 
\end{itemize} 
A graph is {\em $i$-usable} if it admits an $i$-usable cover. Our goal is to show that for all $i \in \{0,\dots,m+1\}$, there exists an $i$-usable graph. Clearly, any complete graph on $d+1$ vertices is $0$-usable. Now, fix some $i \in \{0,\dots,m\}$, and suppose that $G_i$ and $G_i'$ are $i$-usable graphs with $i$-usable covers $(A_1,\dots,A_{2^i})$ and $(A_{2^i+1},\dots,A_{2^{i+1}})$, respectively. In view of (b) and (c), we may assume that for all distinct $j_1,j_2 \in \{1,\dots,2^i\}$, $A_{j_1} \cap A_{2^i+j_2} = \emptyset$, and that for all $j \in \{1,\dots,2^i\}$, $A_j \cap A_{2^i+j}$ is of size $2^{m-i}$ and intersects neither $\partial_{G_i}(A_j)$ nor $\partial_{G_i'}(A_{2^i+j})$. Using part (c), we see that $|\bigcup_{j=1}^{2^i} (A_j \cap A_{2^i+j})| = k$ and $G_i[\bigcup_{j=1}^{2^i} (A_j \cap A_{2^i+j})] = G_i'[\bigcup_{j=1}^{2^i} (A_j \cap A_{2^i+j})]$. Let $G_{i+1}$ be the graph obtained by gluing $G_i$ and $G_i'$ along the set $\bigcup_{j=1}^{2^i} (A_j \cap A_{2^i+j})$. Let us verify that $(A_1,\dots,A_{2^{i+1}})$ is an $(i+1)$-usable cover of $G_{i+1}$. Part (a) is immediate. Next, note that for all $j \in \{1,\dots,2^i\}$, $\partial_{G_{i+1}}(A_j) = \partial_{G_i}(A_j) \cup (A_j \cap A_{2^i+j})$ and $\partial_{G_{i+1}}(A_{2^i+j}) = \partial_{G_i'}(A_{2^i+j}) \cup (A_j \cap A_{2^i+j})$, and furthermore, every vertex in $A_j \cap A_{2^i+j}$ is complete to both $A_j \smallsetminus \partial_{G_{i+1}}(A_j)$ and $A_{2^i+j} \smallsetminus \partial_{G_{i+1}}(A_{2^i+j})$ (we use the fact that $A_j$ and $A_{2^i+j}$ are cliques). Parts (b) and (c) now follow from the induction hypothesis. Finally, since $G_{i+1}$ is obtained by gluing $G_i$ and $G_i'$ along $k$ vertices, any $(k+1)$-connected induced subgraph of $G_{i+1}$ is in fact an induced subgraph of $G_i$ or $G_i'$, and so (d) follows from the induction hypothesis. This completes the induction. 

\begin{figure}
\begin{center}
\includegraphics[scale=0.3]{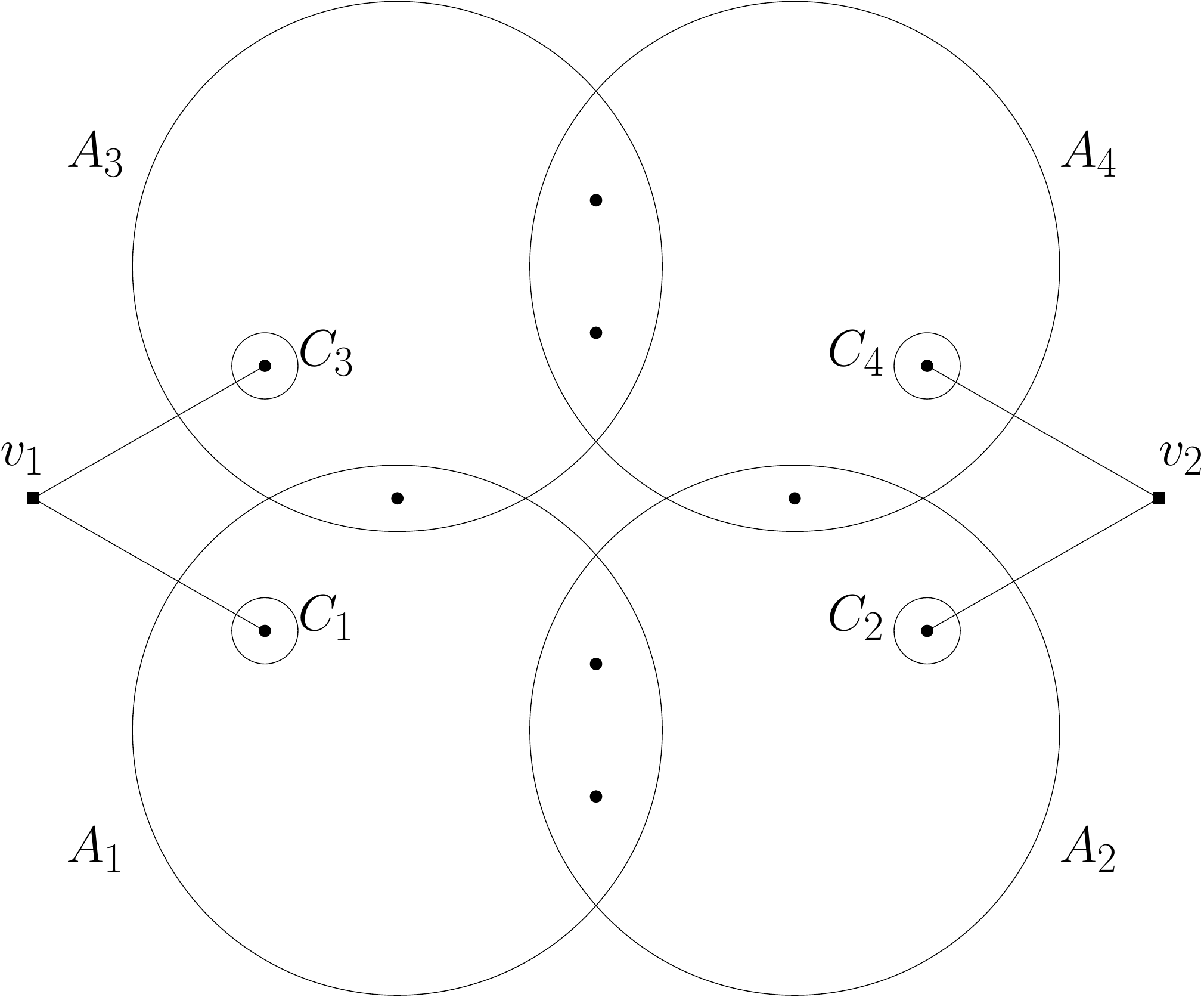} 
\end{center}
\caption{Graph $H$ for the case $k=2$}\label{fig:H-k=2}
\end{figure}

Now, fix an $(m+1)$-usable graph $H_0$, and let $(A_1,\dots,A_{2k})$ be an $(m+1)$-usable cover of $H_0$ (we use the fact that $2^{m+1} = 2k$). Since $V(H_0) = \bigcup_{j=1}^{2k} A_j$, and since $A_1,\dots,A_{2k}$ are cliques of size $d+1$, we know that $\delta(H_0) \geq d$. Next, using (b), for each $j \in \{1,\dots,2k\}$, we choose a set $C_j \subseteq A_j \smallsetminus \partial_{H_0}(A_j)$ of size $\frac{k}{2}$. By (c) then, the sets $C_1,\dots,C_{2k}$ are pairwise disjoint. Let $v_1,\dots,v_k$ be pairwise distinct vertices that do not belong to $V(H_0)$. Let $H$ be the graph with vertex-set $V(H_0) \cup \{v_1,\dots,v_k\}$ such that $H[V(H_0)] = H_0$, and such that for all $j \in \{1,\dots,k\}$, $N_H(v_j) = C_j \cup C_{k+j}$ (see Figure~\ref{fig:H-k=2}). Let us now verify that the graph $H$ satisfies the following two statements: 
\begin{itemize} 
\item[(i)] for all $j \in \{1,\dots,2k\}$, ${\rm deg}_H(v_j) = k$, and for all $v \in V(H_0)$, ${\rm deg}_H(v) \geq d$; 
\item[(ii)] for all non-empty sets $S \subseteq V(H)$, if $H[S]$ is $(k+1)$-connected and $\partial_H(S) \subsetneq S$, then $w_{V(H) \smallsetminus S}^k(\partial_H(S)) \geq 2k^2-\frac{k}{2}$. 
\end{itemize} 
By construction, for all $j \in \{1,\dots,2k\}$, ${\rm deg}_H(v_j) = |C_j \cup C_{k+j}| = k$, and we already saw that $\delta(H_0) \geq d$. This proves (i). For (ii), fix a non-empty set $S \subseteq V(H)$ such that $H[S]$ is $(k+1)$-connected and $\partial_H(S) \subsetneq S$. Since $H[S]$ is $(k+1)$-connected, we know that $\delta(H[S]) \geq k+1$, and so by (i), $v_1,\dots,v_k \notin S$. Thus, $S \subseteq V(H_0)$, and it follows that $H_0[S]$ is $(k+1)$-connected and that $\partial_{H_0}(S) \subseteq \partial_H(S) \subsetneq S$. Since $(A_1,\dots,A_{2k})$ is an $i$-usable cover of $H_0$, it follows that $S$ is equal to one of $A_1,\dots,A_{2k}$; by symmetry, we may assume that $S = A_1$, and so we just need to prove that $w_{V(H) \smallsetminus A_1}^k(\partial_H(A_1)) \geq 2k^2-\frac{k}{2}$. By construction, $\partial_H(A_1) = \partial_{H_0}(A_1) \cup C_1$. From (b), we know that $|\partial_{H_0}(A_1)| = 2k-1$, and that every vertex in $\partial_{H_0}(A_1)$ is $k$-strong with respect to $V(H_0) \smallsetminus A_1$ in $H_0$, and therefore, with respect to $V(H) \smallsetminus A_1$ in $H$ as well. Thus, $$w_{V(H) \smallsetminus A_1}^k(\partial_H(A_1)) \, \geq \, k|\partial_{H_0}(A_1)|+|C_1| \, = \, k(2k-1)+\frac{k}{2} \, = \, 2k^2-\frac{k}{2}.$$ This proves (ii). 

Finally, let $G$ be the graph obtained by gluing $d$ copies of $H$ along the set $\{v_1,\dots,v_k\}$. The fact that $\delta(G) \geq d$ follows from (i). Now, fix a cut-partition $(A,B,C)$ of $G$ such that $G[A \cup C]$ is $(k+1)$-connected. Then $G[A \cup C]$ is in fact an induced subgraph of one of the $d$ copies of $H$ used to construct $G$, and furthermore, $\partial_H(A \cup C) \subseteq \partial_G(A \cup C) \subseteq C \subsetneq A \cup C$. The fact that $w_B^k(C) \geq 2k^2-\frac{k}{2}$ now follows from (ii). 
\end{proof}

Note that in the case $k = 2$, Proposition~\ref{k-weight-power-of-2} implies that the bound of $2k^2-1 = 7$ from Corollary~\ref{cor-main} is best possible. This stands in contrast to the fact that $\psi_c(2) = 5$ (see Theorem~\ref{th:bounds}). So far, we have seen that Corollary~\ref{cor-main} implies Theorems~\ref{th:sf} and~\ref{th:conn}. Using the fact that $\psi_c(2) = 5$, we can obtain a better bound in Theorems~\ref{th:sf} and~\ref{th:conn} for the case $k = 2$. However, Corollary~\ref{cor-main} is still useful for the case $k = 2$ because in Section~\ref{sec:col}, we use it to prove Theorem~\ref{th:col}. Using the fact that $\psi_c(2) = 5$ and an argument analogous to the proof of Proposition~\ref{prop:coloring}, one can show that if $\chi(G) > c+5$, then $G$ contains a $3$-connected induced subgraph of chromatic number greater than $c$. However, for the case $k = 2$, Theorem~\ref{th:col} (whose proof is based on Corollary~\ref{cor-main}) states that the condition $\chi(G) > \max\{c+2,8\}$ is sufficient, and this is clearly better than the condition $\chi(G) > c+5$ as soon as $c \geq 4$.

\section{Proof of Theorem~\ref{th:bounds}} \label{sec:bounds}

In this section, we prove Theorem~\ref{th:bounds}, restated below for the reader's convenience. 

\begin{th:bounds} For all positive integers $k$, all the following hold: 
\begin{itemize} 
\item[(1)] $2k-1 \leq \varphi(k) \leq \psi(k) \leq \psi_c(k) \leq 2k^2-1$; 
\item[(2)] $k^2+k-1 \leq \psi(k) \leq \psi_c(k) \leq 2k^2-1$; 
\item[(3)] $\frac{1}{4}k\log_2k < \varphi(k) \leq 2k^2-1$. 
\end{itemize}
Furthermore, 
\begin{itemize} 
\item[(4)] $\varphi(2) = \psi(2) = \psi_c(2) = 5$. 
\end{itemize} 
\end{th:bounds} 

\medskip 

We begin with a proposition, which is a slightly weaker version of part (1) of Theorem~\ref{th:bounds} (the lower bound of $2k-1$ is replaced by $k$).

\begin{proposition} \label{prop:bounds} For all positive integers $k$, $k \leq \varphi(k) \leq \psi(k) \leq \psi_c(k) \leq 2k^2-1$. 
\end{proposition} 
\begin{proof} 
Fix a positive integer $k$. The inequality $\psi_c(k) \leq 2k^2-1$ follows from Theorem~\ref{th:conn}. The fact that $k \leq \varphi(k)$ is a simple exercise: consider, for instance, the graph $H$ obtained by gluing two complete graphs, each of size $2k$, along a clique of size $k$; the graph $H$ is not complete, it belongs to the $k$-closure of the class of all complete graphs, and for every cut-partition $(A,B,C)$ of $H$ such that $H[A]$ is a complete graph, we have that $|C| \geq k$. For the remaining two inequalities, we fix an arbitrary hereditary class $\mathcal{G}$ and an arbitrary graph $G$ in $\mathcal{G}^k$. 

To show that $\psi(k) \leq \psi_c(k)$, we must show that either $G \in \mathcal{G}$, or $G$ admits a cut-partition $(A,B,C)$ such that $G[A \cup C] \in \mathcal{G}$ and $|C| \leq \psi_c(k)$, or $G$ has a vertex of degree at most $\psi_c(k)$. We may assume that $G$ is not $(k+1)$-connected, for otherwise, Lemma~\ref{l:inG} guarantees that $G \in \mathcal{G}$, and we are done. Likewise, we may assume that $\delta(G) > \psi_c(k)$. Then by the definition of $\psi_c$, there exists a cut-partition $(A,B,C)$ of $G$ such that $G[A \cup C]$ is $(k+1)$-connected and $|C| \leq \psi_c(k)$. Since $G[A \cup C]$ is $(k+1)$-connected, Lemma~\ref{l:inG} implies that $G[A \cup C] \in \mathcal{G}$. This proves that $\psi(k) \leq \psi_c(k)$. 

It remains to show that $\varphi(k) \leq \psi(k)$. For this, we must prove that either $G \in \mathcal{G}$, or $G$ admits a cut-partition $(A,B,C)$ such that $G[A] \in \mathcal{G}$ and $|C| \leq \psi(k)$. We may assume that $G \notin \mathcal{G}$, for otherwise we are done. By the definition of $\psi$ then, we know that $G$ either admits a cut-partition $(A,B,C)$ such that $G[A \cup C] \in \mathcal{G}$ and $|C| \leq \psi(k)$, or contains a vertex of degree at most $\psi(k)$. In the former case, we use the fact that $\mathcal{G}$ is hereditary to deduce that $G[A] \in \mathcal{G}$, and we are done. So assume that $G$ contains a vertex of degree at most $\psi(k)$. Fix a vertex $v$ of $G$ of degree $\delta(G) \leq \psi(k)$. If $v$ has a non-neighbor in $G$, then we set $A = \{v\}$, $B = V(G) \smallsetminus N[v]$, and $C = N(v)$, and we are done. If $v$ has no non-neighbors in $G$, then the fact that ${\rm deg}(v) = \delta(G)$ implies that $G$ is a complete graph, and so Lemma~\ref{l:inG} guarantees that $G \in \mathcal{G}$, and again we are done. 
\end{proof} 

\subsection*{Proof of parts (1) and (2): a linear lower bound for $\varphi$ and a quadratic lower bound for $\psi$} 

In this subsection, we prove parts (1) and (2) of Theorem~\ref{th:bounds}. In view of part Proposition~\ref{prop:bounds}, it suffices to show that for all integers $k \geq 2$, we have that $\varphi(k) \geq 2k-1$ and $\psi(k) \geq k^2+k-1$. We begin with a technical lemma. 

\begin{lemma} \label{l:addk} Let $k$ be a positive integer, let $\cal G$ be a hereditary class that contains all complete graphs of size at most $k+1$, and let $G$ be a graph on at least two vertices. If $G$ contains a vertex $v$ of degree at most $k$ such that $G \smallsetminus v \in \mathcal{G}^k$, then $G \in \mathcal{G}^k$. In other words, if $G$ can be obtained by adding a vertex of degree at most $k$ to a graph that belongs to $\mathcal{G}^k$, then $G$ belongs to $\mathcal{G}^k$. 
\end{lemma}
\begin{proof}
Note first that every graph on at most $k+1$ vertices is either complete (and therefore belongs to $\mathcal{G}$) or admits a cutset of size at most $k-1$ (and can therefore be obtained by gluing two smaller graphs along at most $k-1$ vertices). Thus, an easy induction on the number of vertices implies that every graph on at most $k+1$ vertices belongs to $\mathcal{G}^k$. 

To prove the lemma, suppose that a vertex $v \in V(G)$ satisfies ${\rm deg}_G(v) \leq k$ and $G \smallsetminus v \in \mathcal{G}^k$. Then $|N_G[v]| \leq k+1$, and so by what we just showed, $G[N_G[v]] \in \mathcal{G}^k$. Since $G$ is obtained by gluing $G \smallsetminus v$ and $G[N_G[v]]$ along the set $N_G[v]$ of size at most $k$, it follows that $G \in \mathcal{G}^k$. 
\end{proof}

For the remainder of this subsection, $k \geq 2$ is a fixed integer, and $\cal G$ is the class of all complete graphs. In order to prove that $\varphi(k) \geq 2k-1$ and $\psi(k) \geq k^2+k-1$, it suffices to construct a graph $G$ in $\mathcal{G}^k$ that has the following four properties: 
\begin{itemize} 
\item[(i)] $G$ is not a complete graph (and therefore $G \notin \mathcal{G}$); 
\item[(ii)] $\delta(G) \geq k^2+k-1$; 
\item[(iii)] $G$ does not admit a cut-partition $(A,B,C)$ such that $G[A \cup C] \in \mathcal{G}$; 
\item[(iv)] for every cut-partition $(A,B,C)$ of $G$ such that $G[A] \in \mathcal{G}$, we have that $|C| \geq 2k-1$. 
\end{itemize} 
Once we have constructed such a graph $G$, (i)-(iii) will imply that $\psi(k) \geq k^2+k-1$, and (i) and (iv) will imply that $\varphi(k) \geq 2k-1$. 

Our construction is as follows. Let $A_1, \dots, A_{k+1}$ be pairwise disjoint sets such that $|A_1| = \dots = |A_k| = k^2+k-1$ and $|A_{k+1}| 
= k^2+k$. For all $j \in \{1,\dots,k+1\}$, let $A_{j,1},\dots,A_{j,k}$ be pairwise disjoint subsets of $A_j$, each of size $k$. For all $j \in \{1,..,k+1\}$, let $C_j = A_j \smallsetminus (\bigcup_{l=1}^k A_{j,l})$. Note that $|C_j| = k-1$ for all $j \in \{1,\dots,k\}$, but $|C_{k+1}| = k$. For all $j \in \{1,\dots,k+1\}$, we set $C_j = \{c_j^1,\dots,c_j^{j-1},c_j^{j+1},\dots,c_j^k\}$. Let $v_1,\dots,v_k,u_1,\dots,u_k$ be pairwise distinct vertices, none of which belongs to $\bigcup_{j=1}^{k+1} A_j$, and set $S = \{u_1,\dots,u_k\}$. Let $H$ be the graph with vertex-set $(\bigcup_{j=1}^{k+1} A_j) \cup \{v_1,\dots,v_k\} \cup \{u_1,\dots,u_k\}$ and adjacency as follows:
\begin{itemize} 
\item $A_1,\dots,A_{k+1}$ are cliques, pairwise anti-complete to each other; 
\item for all $l \in \{1,\dots,k\}$, $N(v_l) = \bigcup_{j=1}^{k+1} A_{j,l}$; 
\item for all $j \in \{1,\dots,k\}$, $N(u_j) = \{c_1^j,\dots,c_{j-1}^j,c_{j+1}^j,\dots,c_{k+1}^j\}$. 
\end{itemize} 
Note that this means that for all $j \in \{1,\dots,k+1\}$, $N_H(A_j) = \{v_1,\dots,v_k\} \cup \{u_1,\dots,u_{j-1},u_{j+1},\dots,u_k\}$, and so $|N_H(A_j)| \geq 2k-1$. 

Let us check that $H \in {\cal G}^k$. First, for all $j\in \{1,\dots,k+1\}$, $H[A_j]$ is a complete graph, so clearly $H[A_j] \in {\cal G}^k$. By applying Lemma~\ref{l:addk} $k$ times, we see that the $H[A_j \cup \{v_1, \dots v_k\}]$'s are all in ${\cal G}^k$, and the graph $H' = H[A_1 \cup \dots \cup A_{k+1} \cup \{v_1, \dots, v_k\}]$ is obtained by gluing these graphs along $\{v_1, \dots, v_k\}$. Consequently $H' \in {\cal G}^k$. Now, we see that $H$ is obtained from $H'$ by successively adding degree $k$ vertices $u_1, \dots, u_k$, and so by Lemma~\ref{l:addk}, $H \in {\cal G}^k$. Let us now verify that $H$ satisfies the following two statements: 
\begin{itemize} 
\item[(a)] $H$ contains no simplicial vertices (a vertex is {\em simplicial} if its neighborhood is a clique), every vertex in $V(H) \smallsetminus S$ is of degree at least $k^2+k-1$, and $S$ is stable set, each of whose vertices is of degree $k$; 
\item[(b)] every clique $K$ of $H$ of size at least $k+2$ satisfies $|N_H(K)| \geq 2k-1$. 
\end{itemize} 
Statement (a) is immediate from the construction of $H$ and the fact that $k \geq 2$. For (b), fix a clique $K$ of $H$ of size at least $k+2$. By construction, none of $v_1,\dots,v_k,u_1,\dots,u_k$ belongs to a clique of size greater than $k+1$, and the cliques $A_1,\dots,A_{k+1}$ are anti-complete to each other. Thus, $K \subseteq A_j$ for some $j \in \{1,\dots,k+1\}$. Since $A_j$ is a clique, and since no vertex in $A_j$ has more than one neighbor in $V(H) \smallsetminus A_j$, we deduce that $|N_H(K)| \geq |N_H(A_j)| \geq 2k-1$, and (b) follows. 

Now, we build $G$ by taking $k+1$ copies of $H$ and gluing them along $S$ (so $G\in {\cal G}^k$). Since $H$ contains no simplicial vertices, neither does $G$. Furthermore, by construction, $\delta(G) \geq k^2 + k -1$, and so (ii) holds. Clearly, $G$ is not a complete graph, and so (i) holds. Furthermore, $G$ contains no cut-partition $(A,B,C)$ such that $G[A \cup C] \in \mathcal{G}$, for otherwise, $A \cup C$ would be a clique, and any vertex of $A$ would be simplicial. Thus, (iii) holds. To prove (iv), we fix a cut-partition $(A,B,C)$ of $G$ such that $G[A] \in \mathcal{G}$ (thus, $A$ is a clique); we need to show that $|C| \geq 2k-1$. Note first that $|A \cup C| \geq |N_G[A]| \geq \delta(G)+1 \geq k^2+k$. Thus, if $|A| \leq k+1$, it follows that $|C| \geq k^2-1 \geq 2k-1$ (we use the fact that $k \geq 2$), and we are done. So assume that $|A| \geq k+2$. Since $A$ is a clique of $G$, we know that $A$ is a clique of one of the $k+1$ copies of $H$ used to construct $G$, and so by (b), $|C| \geq |N_G(A)| \geq |N_H(A)| \geq 2k-1$. This proves (iv), and we are done.

\subsection*{Proof of part (3): a superlinear lower bound for $\varphi$}

In this subsection, we prove part (3) of Theorem~\ref{th:bounds}. In view of Proposition~\ref{prop:bounds}, it suffices to show that that for all positive integers $k$, $\varphi(k) > \frac{1}{4}k\log_2k$. We prove this in two stages. We first prove Lemma~\ref{lower-phi-construction} (stated below), which gives a slightly better lower bound for $\varphi(k)$ in the case when $k$ is a power of~$2$. We prove Lemma~\ref{lower-phi-construction} by constructing a suitable graph that belongs to the $k$-closure of the class of all complete graphs. We then use Lemma~\ref{lower-phi-construction} and the fact that $\varphi$ is non-decreasing to prove that $\frac{1}{4}k \log_2k < \varphi(k)$ (see Lemma~\ref{lb-phi-all-k}).

\begin{lemma} \label{lower-phi-construction} Let $m$ be a positive integer, and let $k = 2^m$. Then $\varphi(k) \geq \frac{1}{2}mk+k-1$. 
\end{lemma} 
\begin{proof} 
First, note that $\sum_{l=1}^m (\frac{k}{2}+\frac{k}{2^l}) = \frac{1}{2}mk+k-1$. Given $i \in \{0,\dots,m\}$, an {\em $i$-good partition} of a graph $G$ is a partition $(A_1,\dots,A_{2^i})$ of $V(G)$ such that the following hold: 
\begin{itemize} 
\item[(a)] for all $j \in \{1,\dots,2^i\}$, $A_j$ is a clique of $G$ of size $\sum_{l=1}^m (\frac{k}{2}+\frac{k}{2^l}) = \frac{1}{2}mk+k-1$; 
\item[(b)] for all $j \in \{1,\dots,2^i\}$, $|\partial_G(A_j)| = \sum_{l=1}^i (\frac{k}{2}+\frac{k}{2^l})$; 
\item[(c)] for all $j \in \{1,\dots,2^i\}$, and all non-empty $K \subseteq A_j$, $|N_G(K)| \geq \sum_{l=1}^i (\frac{k}{2}+\frac{k}{2^l})$; 
\item[(d)] every clique of $G$ of size greater than $k$ is included in one of $A_1,\dots,A_{2^i}$. 
\end{itemize} 
\noindent 
Given $i \in \{1,\dots,m\}$, we say that a graph is {\em $i$-good} provided it admits an $i$-good partition. 

Let $\mathcal{G}$ be the class of all complete graphs. Let us first show that if $\mathcal{G}^k$ contains an $m$-good graph, then $\varphi(k) \geq \frac{1}{2}mk+k-1$. Indeed, suppose that $G \in \mathcal{G}^k$ is an $m$-good graph, and let $(A_1,\dots,A_{2^m})$ be an $m$-good partition of $G$. By (a), $A_1,\dots,A_{2^m}$ are cliques of $G$, each of size $\frac{1}{2}mk+k-1$. By (d) then, $G$ is not a complete graph, and so $G \notin \mathcal{G}$. Let $(A,B,C)$ be a cut-partition of $G$ such that $G[A] \in \mathcal{G}$; we need to show that $|C| \geq \frac{1}{2}mk+k-1$. Since $N_G(A) \subseteq C$, it suffices to show that $|N_G(A)| \geq \frac{1}{2}mk+k-1$. Since $G[A] \in \mathcal{G}$, we know that $A$ is a clique of $G$. If $A$ is included in one of $A_1,\dots,A_{2^m}$, then by (c), we have that $|N_G(A)| \geq \sum_{l=1}^m (\frac{k}{2}+\frac{k}{2^l}) = \frac{1}{2}mk+k-1$, and we are done. So suppose that $A$ intersects at least two of $A_1,\dots,A_{2^m}$, say $A_1$ and $A_2$. By (d) then, $|A| \leq k$. Since $A_1$ and $A_2$ are cliques, both of which intersect $A$, we know that $(A_1 \cup A_2) \smallsetminus A \subseteq N_G(A)$; consequently, $|N_G(A)| \geq |A_1|+|A_2|-|A| \geq 2(\frac{1}{2}mk+k-1)-k \geq \frac{1}{2}mk+k-1$, and we are done. This reduces the problem to showing that $\mathcal{G}^k$ contains an $m$-good graph. We do this inductively, that is, we show that $\mathcal{G}^k$ contains an $i$-good graph for all $i \in \{0,\dots,m\}$. 

Clearly, any complete graph on $\frac{1}{2}mk+k-1$ vertices is $0$-good and belongs to $\mathcal{G}^k$. Next, fix $i \in \{0,\dots,m-1\}$, and suppose that $\mathcal{G}^k$ contains an $i$-good graph. Using the fact that $\mathcal{G}^k$ is closed under isomorphism, we fix isomorphic $i$-good graphs $G_i,G_i' \in \mathcal{G}^k$ on disjoint vertex-sets. Let $(A_1,\dots,A_{2^i})$ and $(A_{2^i+1},\dots,A_{2^{i+1}})$ be $i$-good partitions of $G_i$ and $G_i'$, respectively. Then for all $j \in \{1,\dots,2^i\}$, we have that $|A_j \smallsetminus \partial_{G_i}(A_j)| = \sum_{l=i+1}^m (\frac{k}{2}+\frac{k}{2^l}) \geq (2^i+1)\frac{k}{2^{i+1}}$, and we let $B_j,C_j^{2^i+1},\dots,C_j^{2^{i+1}}$ be pairwise disjoint subsets of $A_j \smallsetminus \partial_{G_i}(A_j)$, each of size $\frac{k}{2^{i+1}}$. Similarly, for all $j \in \{1,\dots,2^i\}$, we let $B_{2^i+j},C_{2^i+j}^1,\dots,C_{2^i+j}^{2^i}$ be pairwise disjoint subsets of $A_{2^i+j} \smallsetminus \partial_{G_i'}(A_{2^i+j})$, each of size $\frac{k}{2^{i+1}}$. (Graphs $G_i$ and $G_i'$ are represented in Figure~\ref{fig-phiini}.) Let $G_{i+1}$ be the graph on the vertex-set $\bigcup_{j=1}^{2^{i+1}} A_j$, with adjacency as follows: 
\begin{itemize} 
\item $G_{i+1}[\bigcup_{j=1}^{2^i} A_j] = G_i$ and $G_{i+1}[\bigcup_{j=1}^{2^i} A_{2^i+j}] = G_i'$; 
\item for all $j \in \{1,\dots,2^i\}$, $B_j$ is complete to $\bigcup_{l=1}^{2^i} C_{2^i+l}^j$, and $B_{2^i+j}$ is complete to $\bigcup_{l=1}^{2^i} C_l^{2^i+j}$; 
\item there are no other edges between $\bigcup_{j=1}^{2^i} A_j$ and $\bigcup_{j=1}^{2^i} A_{2^i+j}$ in $G_{i+1}$. 
\end{itemize} 
\noindent 
(The graph $G_{i+1}$ is represented in Figure~\ref{fig-phifinal}.) We claim that $G_{i+1}$ is an $(i+1)$-good graph, and that $G_{i+1} \in \mathcal{G}^k$. We first prove the latter. For all $j \in \{1,\dots,2^i\}$, let $H_{2^i+j} = G_{i+1}[B_{2^i+j} \cup \bigcup_{l=1}^{2^i} C_l^{2^i+j}]$; then $H_{2^i+j}$ is obtained by gluing complete graphs on the vertex-sets $B_{2^i+j} \cup C_1^{2^i+j},\dots,B_{2^i+j} \cup C_{2^i}^{2^i+j}$ along the set $B_{2^i+j}$ of size $\frac{k}{2^{i+1}} < k$, and consequently, $H_{2^i+j} \in \mathcal{G}^k$. Next, set $\widetilde{G}_i = G_{i+1}[\bigcup_{j=1}^{2^i} (A_j \cup B_{2^i+j})]$ and $\widetilde{G}_i' = G_{i+1}[\bigcup_{j=1}^{2^i} (A_{2^i+j} \cup B_j)]$ (see Figure~\ref{fig-phimedleft}). Clearly, the graph $\widetilde{G}_i$ is obtained from $G_i$ by sequentially gluing the graphs $H_{2^i+1},\dots,H_{2^{i+1}}$ along the sets $\bigcup_{l=1}^{2^i} C_l^{2^i+1},\dots,\bigcup_{l=1}^{2^i} C_l^{2^{i+1}}$, respectively. Since $G_i,H_{2^i+1},\dots,H_{2^{i+1}} \in \mathcal{G}^k$, and since $|\bigcup_{l=1}^{2^i} C_l^j| = \frac{k}{2} < k$ for all $j \in \{1,\dots,2^i\}$, it follows that $\widetilde{G}_i \in \mathcal{G}^k$. Similarly, $\widetilde{G}_i' \in \mathcal{G}^k$. But $G_{i+1}$ is obtained by gluing $\widetilde{G}_i$ and $\widetilde{G}_i'$ along the set $\bigcup_{j=1}^{2^{i+1}} B_j$, and $|\bigcup_{j=1}^{2^{i+1}} B_j| = k$. This proves that $G_{i+1} \in \mathcal{G}^k$. 

\begin{figure}
\begin{center}
\includegraphics[scale=0.7]{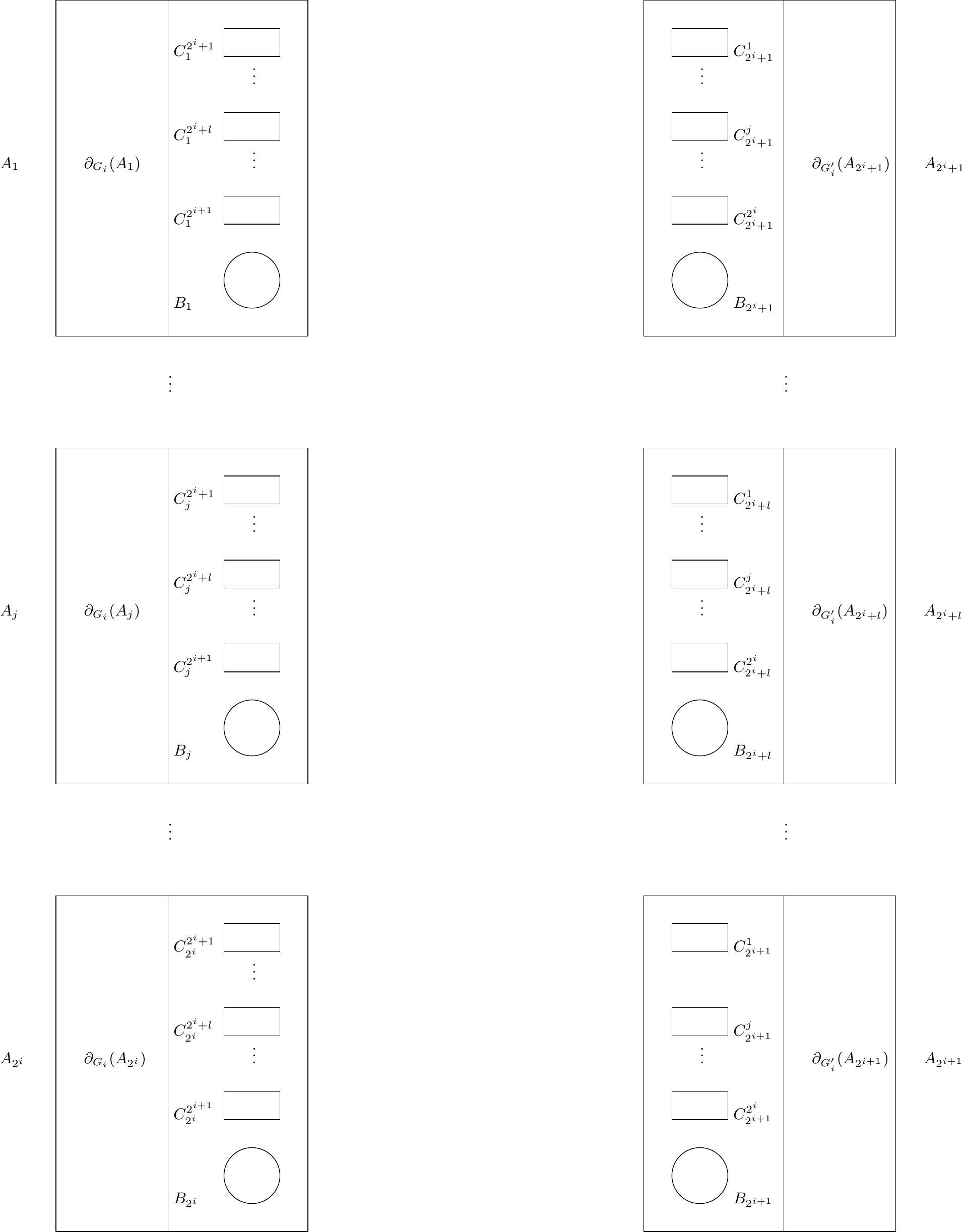} 
\end{center}
\caption{Graphs $G_i$ (left) and $G_i'$ (right)}\label{fig-phiini}
\end{figure}

\begin{figure}
\begin{center}
\includegraphics[scale=0.7]{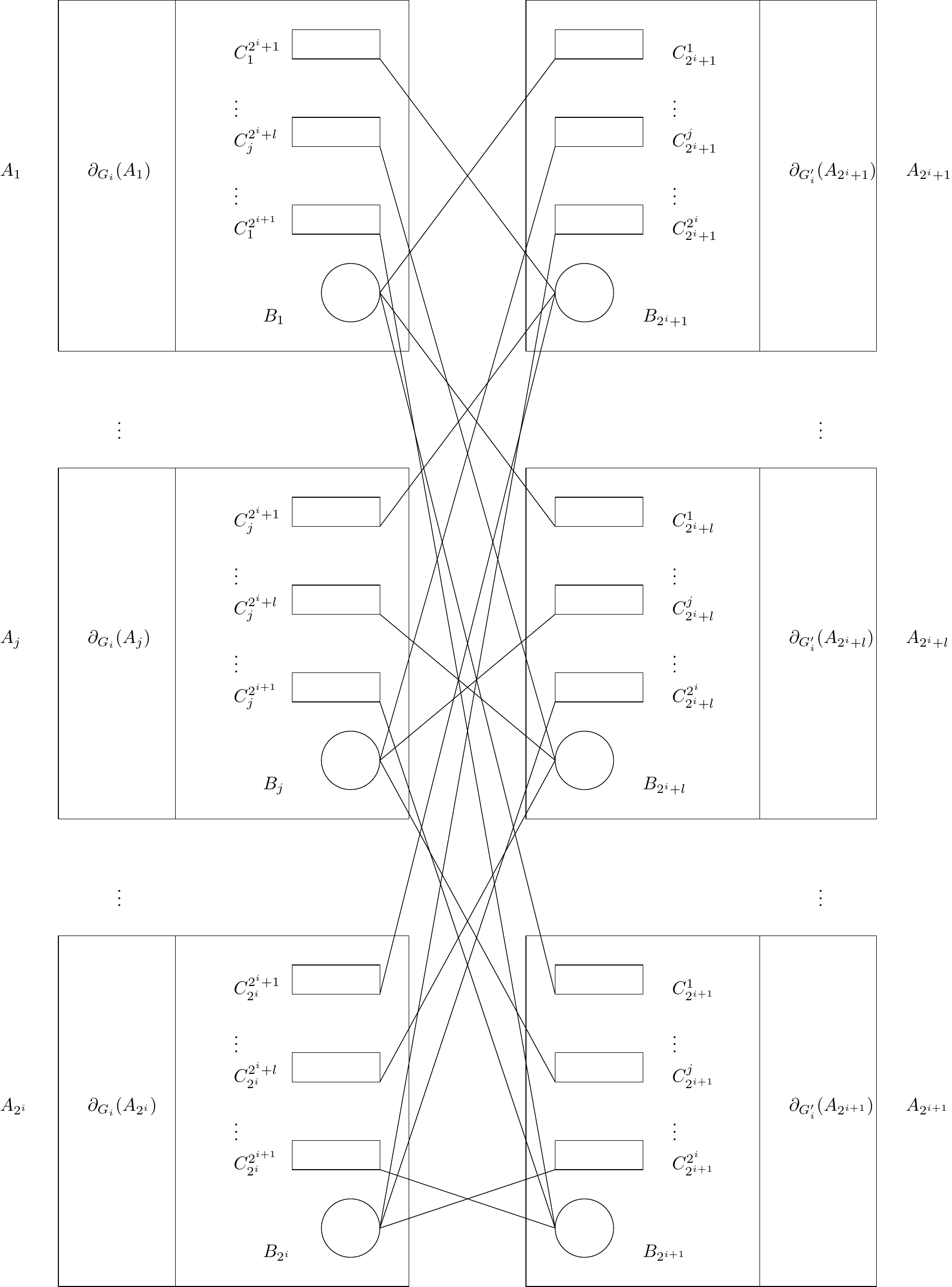}
\end{center}
\caption{Graph $G_{i+1}$}\label{fig-phifinal}
\end{figure}

\begin{figure}
\begin{center}
\includegraphics[scale=0.7]{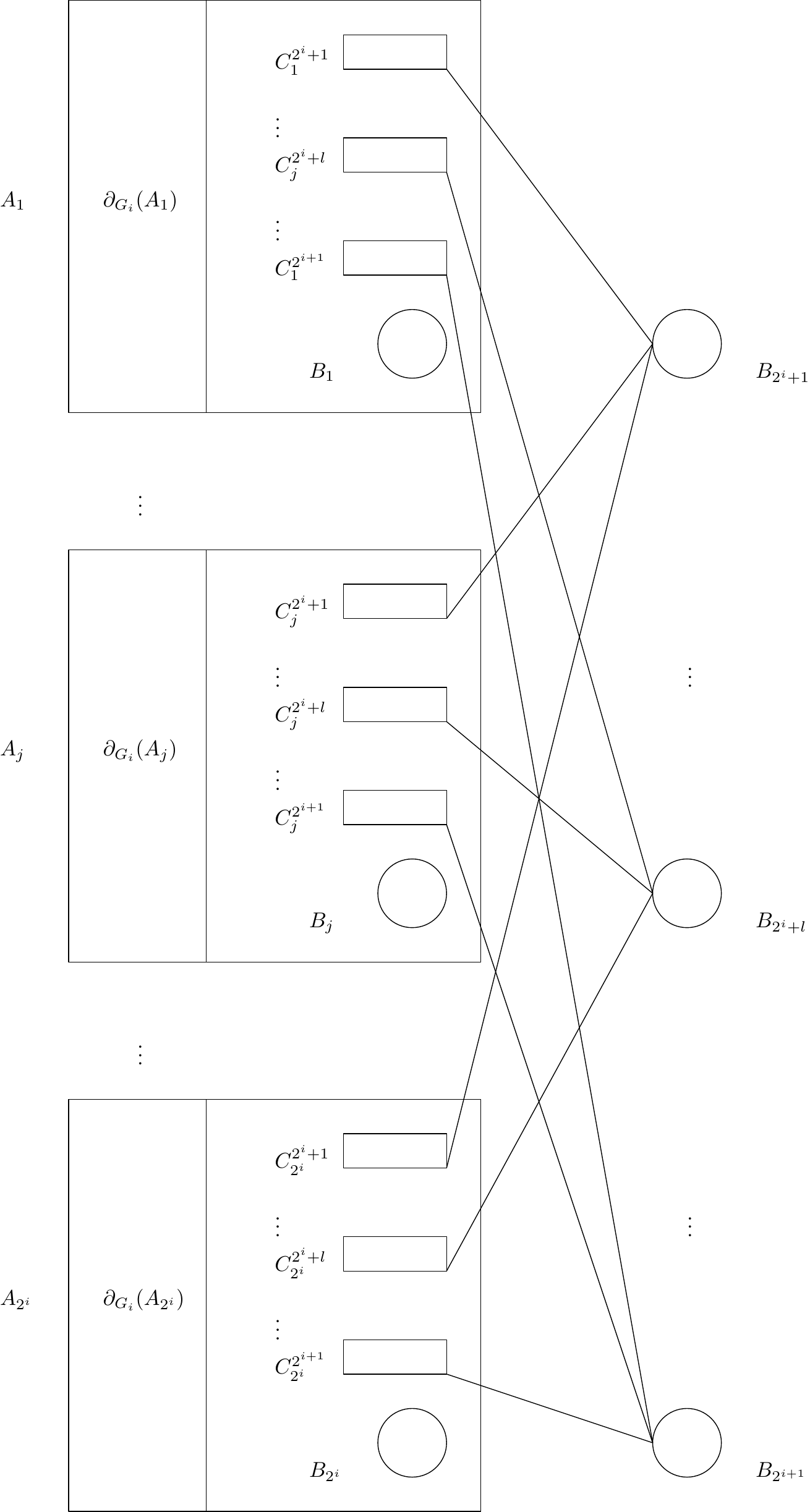}
\end{center}
\caption{Graph $\widetilde{G}_i$}\label{fig-phimedleft}
\end{figure}

It remains to show that $G_{i+1}$ is an $(i+1)$-good graph. We do this by showing that $(A_1,\dots,A_{2^{i+1}})$ is an $(i+1)$-good partition of $G_{i+1}$. Since $(A_1,\dots,A_{2^i})$ and $(A_{2^i+1},\dots,A_{2^{i+1}})$ are $i$-good partitions of $G_i$ and $G_i'$, respectively, we immediately deduce that $(A_1,\dots,A_{2^{i+1}})$ is a partition of $V(G_{i+1})$, and that for all $j \in \{1,\dots,2^{i+1}\}$, $A_j$ is a clique of $G_{i+1}$ of size $\frac{1}{2}mk+k-1$. Thus, (a) holds. To prove (b), we observe that for all $j \in \{1,\dots,2^i\}$, we have that $\partial_{G_{i+1}}(A_j) = \partial_{G_i}(A_j) \cup B_j \cup \bigcup_{l=1}^{2^i} C_j^{2^i+l}$, and since $|\partial_{G_i}(A_j)| = \sum_{l=1}^i (\frac{k}{2}+\frac{k}{2^l})$ and $|B_j| = |C_j^{2^i+1}| = \dots = |C_j^{2^{i+1}}| = \frac{k}{2^{i+1}}$, it follows that $|\partial_{G_{i+1}}(A_j)| = \sum_{l=1}^{i+1} (\frac{k}{2}+\frac{k}{2^l})$. Similarly, $|\partial_{G_{i+1}}(A_{2^i+j})| = \sum_{l=1}^{i+1} (\frac{k}{2}+\frac{k}{2^l})$ for all $j \in \{1,\dots,2^i\}$. Thus, (b) holds. We next prove (d). Suppose that $K$ is a clique of $G_{i+1}$ that intersects both $V(G_i)$ and $V(G_i')$. From the construction of $G_{i+1}$, we deduce that there exists for some $j \in \{1,\dots,2^i\}$ such that either $K \subseteq B_j \cup \bigcup_{l=1}^{2^i} C_{2^i+l}^j$ or $K \subseteq B_{2^i+j} \cup \bigcup_{l=1}^{2^i} C_l^{2^i+j}$. But $|B_j \cup \bigcup_{l=1}^{2^i} C_{2^i+l}^j| = |B_{2^i+j} \cup \bigcup_{l=1}^{2^i} C_l^{2^i+j}| = (2^i+1)\frac{k}{2^{i+1}} \leq k$ for all $j \in \{1,\dots,2^i\}$, and so $|K| \leq k$. This proves that every clique of $G_{i+1}$ of size greater than $k$ is included in one of $V(G_i)$ and $V(G_i')$; since $(A_1,\dots,A_{2^i})$ and $(A_{2^i+1},\dots,A_{2^{i+1}})$ are $i$-good partitions of $G_i$ and $G_i'$, respectively, it follows that every clique of $G_{i+1}$ of size greater than $k$ is included in one of $A_1,\dots,A_{2^{i+1}}$. This proves (d). 

It remains to prove (c). By symmetry, it suffices to show that for all non-empty $K \subseteq A_1$, we have that $|N_{G_{i+1}}(K)| \geq \sum_{l=1}^{i+1} (\frac{k}{2}+\frac{k}{2^l})$. Fix a non-empty set $K \subseteq A_1$, and set $\widetilde{K} = K \cup B_1 \cup \bigcup_{j=1}^{2^i} C_1^{2^i+j}$. Clearly, $K \subseteq \widetilde{K} \subseteq A_1$. Since $(A_1,\dots,A_{2^i})$ is an $i$-good partition of $G_i$, we know that $|N_{G_i}(\widetilde{K})| \geq \sum_{l=1}^i (\frac{k}{2}+\frac{k}{2^l})$. Next, if $B_1 \cap K = \emptyset$, then set $\widetilde{B}_1 = B_1$, and otherwise, set $\widetilde{B}_1 = (B_1 \smallsetminus K) \cup (\bigcup_{j=1}^{2^i} C_{2^i+j}^1)$. For all $j \in \{1,\dots,2^i\}$, if $K \cap C_1^{2^i+j} = \emptyset$, then set $\widetilde{C}_1^{2^i+j} = C_1^{2^i+j}$, and otherwise, set $\widetilde{C}_1^{2^i+j} = (C_1^{2^i+j} \smallsetminus K) \cup B_{2^i+j}$. Clearly, $|\widetilde{B}_1| \geq \frac{k}{2^{i+1}}$, and for all $j \in \{1,\dots,2^i\}$, $|\widetilde{C}_1^{2^i+j}| \geq \frac{k}{2^{i+1}}$. It is also clear that the sets $\widetilde{B}_1,\widetilde{C}_1^{2^i+1},\dots,\widetilde{C}_1^{2^{i+1}}$ are pairwise disjoint, and so it follows that $|\widetilde{B}_1 \cup \bigcup_{j=1}^{2^i} \widetilde{C}_1^{2^i+j}| \geq \frac{k}{2}+\frac{k}{2^{i+1}}$. Further, it is easy to see that $N_{G_{i+1}}(K) = N_{G_i}(\widetilde{K}) \cup \widetilde{B}_1 \cup \bigcup_{j=1}^{2^i} \widetilde{C}_1^{2^i+j}$. Since $|N_{G_i}(\widetilde{K})| \geq \sum_{l=1}^i (\frac{k}{2}+\frac{k}{2^l})$ and $|\widetilde{B}_1 \cup \bigcup_{j=1}^{2^i} \widetilde{C}_1^{2^i+j}| \geq \frac{k}{2}+\frac{k}{2^{i+1}}$, it follows that $|N_{G_{i+1}}(K)| \geq \sum_{l=1}^{i+1} (\frac{k}{2}+\frac{k}{2^l})$. This proves (c), and we are done. 
\end{proof} 
\begin{lemma} \label{lb-phi-all-k} For all positive integers $k$, $\varphi(k) > \frac{1}{4}k\log_2k$. 
\end{lemma} 
\begin{proof} 
Clearly, the function $\varphi$ is non-decreasing, and by Proposition~\ref{prop:bounds}, we have that $\varphi(k) \geq k$ for all positive integers $k$. Now, fix a positive integer $k$. If $k \leq 15$, then $\frac{1}{4}k\log_2k < k \leq \varphi(k)$, and we are done. So assume that $k \geq 16$. Fix a positive integer $m$ such that $2^m \leq k < 2^{m+1}$; then $k \geq 2^m > \frac{k}{2}$. We now have the following: 
\begin{displaymath} 
\begin{array}{rclll} 
\varphi(k) & \geq & \varphi(2^m) & & \text{since $\varphi$ is non-decreasing} 
\\
\\
& \geq & \frac{1}{2}m2^m+2^m-1 & & \text{by Lemma~\ref{lower-phi-construction}} 
\\
\\
& > & \frac{1}{2}(\frac{k}{2})\log_2(\frac{k}{2})+\frac{k}{2}-1 & & \text{since $2^m > \frac{k}{2}$} 
\\
\\
& = & \frac{1}{4}k\log_2(k)+\frac{k}{4}-1 & & 
\\
\\
& > & \frac{1}{4}k\log_2k & & \text{since $k \geq 16$.} 
\end{array} 
\end{displaymath} 
\noindent 
This completes the argument. 
\end{proof}

\subsection*{Proof of part (4): the case $k = 2$} 

In this subsection, we prove part (4) of Theorem~\ref{th:bounds}, which states that $\varphi(2) = \psi(2) = \psi_c(2) = 5$. In view of Proposition~\ref{prop:bounds}, it suffices to prove the following two inequalities: $\varphi(2) \geq 5$ and $\psi_c(2) \leq 5$. 

\medskip

Let us first show that $\varphi(2) \geq 5$. We denote by $K_6 \smallsetminus e$ 
the graph obtained by deleting one edge from the complete graph on six vertices 
($K_6 \smallsetminus e$ is the graph $G_4$ represented in Figure~\ref{fig:G45}). 
Let $\mathcal{G}$ be the class of all induced subgraphs of $K_6\sm e$. Graphs
$G_1, \dots, G_9$ are represented in Figures~\ref{fig:G12}--\ref{fig:G89}. 
They are all taken from $\cal G$, or obtained from previous graphs by adding a 
vertex of degree two (by Lemma~\ref{l:addk}, this preserves membership in 
${\cal G}^2$), or by gluing previously constructed graphs along two vertices. 
The only non-empty sets $X \subseteq V(G_9)$ such that $G_9[X] \in \mathcal{G}$ 
and $|N(X)| < 5$ are subsets of $\{v, w, x\}$. Let us build a copy of $G_9$, and 
obtain a new graph $G_{10}$ by gluing $G_9$ and its copy along $\{v,w\}$. There 
are two vertices arising from $x$, and we name them $x$ and $x'$. The only 
non-empty sets $X \subseteq V(G_{10})$ such that $G_{10}[X] \in {\cal G}$ and 
$|N(X)| < 5$ are subsets of $\{x, x'\}$. We now build a new graph $G_{11}$ by 
gluing three copies of $G_{10}$ along $\{x, x'\}$. In this graph, no non-empty 
set $X$ of vertices such that $G_{11}[X] \in \mathcal{G}$ satisfies $|N(X)| < 5$. 
This proves that $\varphi(2) \geq 5$.

\begin{figure}
\center
\includegraphics[height=2cm]{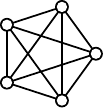}
\hspace{3em}
\includegraphics[height=2cm]{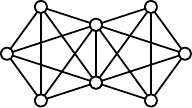}
\hspace{3em}
\includegraphics[height=2cm]{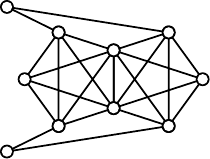}
\caption{Graphs $G_1$ (left), $G_2$ (middle), and $G_3$ (right) \label{fig:G12}}
\end{figure}

\begin{figure}
\center
\includegraphics[height=3cm]{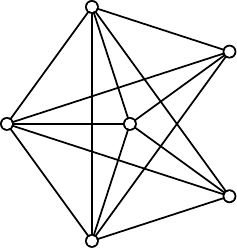}
\hspace{3em}
\includegraphics[height=3cm]{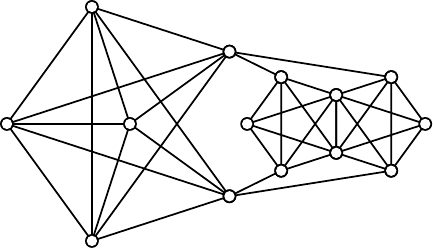}
\caption{Graphs $G_4$ (left) and $G_5$ (right) \label{fig:G45}}
\end{figure}

\begin{figure}
\center
\includegraphics[width=5cm]{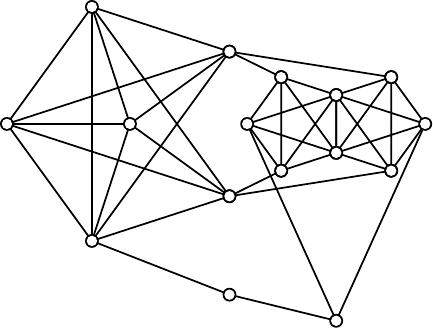}
\hspace{1em}
\includegraphics[width=5cm]{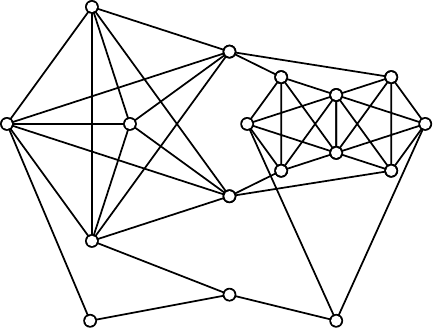}
\caption{Graphs $G_6$ (left) and $G_7$ (right) \label{fig:G67}}
\end{figure}

\begin{figure}
\center
\includegraphics[width=5cm]{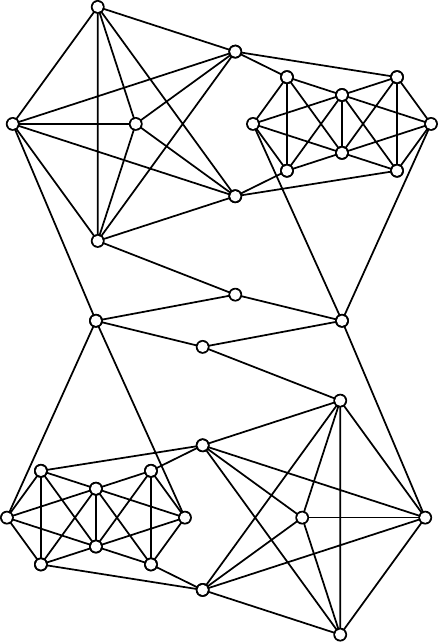}
\hspace{1em}
\includegraphics[width=5cm]{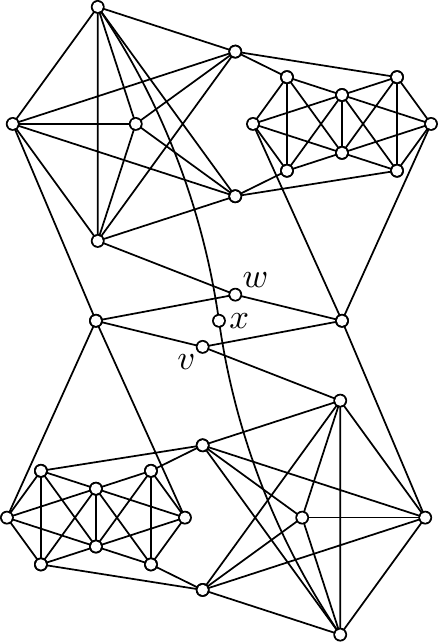}
\caption{Graphs $G_8$ (left) and $G_9$ (right) \label{fig:G89}}
\end{figure}

\medskip

\begin{table}[h]\center
\begin{tabular}{rc|rcccccc}
& {$|Y|$}              & $\geq 6$  & $5$ & $4$ & $3$ & $2$ & $1$   & $0$           \\ 
{$|X|$}              & &                 &     &     &     &       &       &       \\ \hline
$0 \phantom{i}$      & & $0 \phantom{i}$ & $0$ & $0$ & $0$ & $0$   & $0$   & $0$   \\
$1 \phantom{i}$      & & $0 \phantom{i}$ & $1$ & $1$ & $1.5$ & $2$ & $2$   & $2$   \\
$2 \phantom{i}$      & & $0 \phantom{i}$ & $1$ & $2$ & $2$ & $2.5$ & $2.5$ & $2.5$ \\
$3 \phantom{i}$      & & $0 \phantom{i}$ & $1$ & $2$ & $3$ & $3$   & $3$   & $3$   \\
$4 \phantom{i}$      & & $0 \phantom{i}$ & $1$ & $2$ & $3$ & $4$   & $4$   & $4$   \\
$\geq 5 \phantom{i}$ & & $0 \phantom{i}$ & $1$ & $2$ & $3$ & $4$   & $4$   & $6$   \\
\end{tabular}
\caption{Power of $(X,Y)$ \label{t:wc}}
\end{table}

It remains to show that $\psi_c(2) \leq 5$. We first need a couple of definitions. Given finite sets $X$ and $Y$, the {\em power} of the ordered pair $(X,Y)$, denoted by $p(X,Y)$, depends on the size of $X$ and $Y$, as shown in Table~\ref{t:wc}. Note that $p(X,Y)$ increases as $|X|$ increases and as $|Y|$ decreases. Furthermore, note that if $p(X,Y) > 0$, then $X \neq \emptyset$ and $|Y| \leq 5$. 

A {\em chunk} of a graph $G$ is an ordered pair $(X,Y)$ of disjoint subsets of $V(G)$ such that $N_G(X) \subseteq Y$ (equivalently: $N_G[X] \subseteq X \cup Y$). Note that if $(X,Y)$ is a chunk of a graph $G$ such that $p(X,Y) > 0$, then $X \neq \emptyset$ and $|N_G(X)| \leq |Y| \leq 5$.

We begin with two technical lemmas.

\begin{lemma} \label{delete-C} Let $(X_1,Y_1),\dots,(X_t,Y_t)$ be chunks of a graph $G$ such that the sets $X_1,\dots,X_t$ are pairwise disjoint, and let $C \subseteq V(G)$ be such that $|C| \leq 2$. For all $j \in \{1,\dots,t\}$, set $\widetilde{X}_j = X_j \smallsetminus C$ and $\widetilde{Y}_j = Y_j \cup (X_j \cap C)$. Then $(\widetilde{X}_1,\widetilde{Y}_1),\dots,(\widetilde{X}_t,\widetilde{Y}_t)$ are chunks of $G$, and the sets $\widetilde{X}_1,\dots,\widetilde{X}_t$ are pairwise disjoint. Furthermore, if $\sum_{j=1}^t p(X_j,Y_j) \geq 6$, then at least one of the following holds: 
\begin{enumerate}[(i)] 
\item $\sum_{j=1}^t p(\widetilde{X}_j,\widetilde{Y}_j) \geq 4$; 
\item $C$ contains a vertex of degree $3$, and $\sum_{j=1}^t p(\widetilde{X}_j,\widetilde{Y}_j) \geq 3.5$; 
\item $C$ contains a vertex of degree at most $2$, and $\sum_{j=1}^t p(\widetilde{X}_j,\widetilde{Y}_j) \geq 3$; 
\item there exists some $i \in \{1,\dots,t\}$ such that $|X_i| = 5$, $|Y_i| = 0$, $C \subseteq X_i$, $|C| = 2$, the two vertices of $C$ are adjacent, and $\sum_{j=1}^t p(\widetilde{X}_j,\widetilde{Y}_j) \geq 3$; 
\item $|C| = 2$, each vertex of $C$ is of degree at most $3$, and $\sum_{j=1}^t p(\widetilde{X}_j,\widetilde{Y}_j) \geq 3$; 
\item $|C| = 2$, one vertex of $C$ is of degree at most $2$ and the other of degree at most $3$, and $\sum_{j=1}^t p(\widetilde{X}_j,\widetilde{Y}_j) \geq 2.5$; 
\item $|C| = 2$, each vertex of $C$ is of degree at most $2$, and $\sum_{j=1}^t p(\widetilde{X}_j,\widetilde{Y}_j) \geq 2$. 
\end{enumerate} 
\end{lemma} 
\begin{proof} 
We may assume that $C \subseteq \bigcup_{j=1}^t X_j$, for otherwise, we set $C' = C \cap (\bigcup_{j=1}^t X_j)$, and we consider the set $C'$ instead of $C$. By construction, for all $j \in \{1,\dots,t\}$, we have that $\widetilde{X}_j \cup \widetilde{Y}_j = X_j \cup Y_j$, $\widetilde{X}_j \cap \widetilde{Y}_j = \emptyset$, and $\widetilde{X}_j \subseteq X_j$, and it follows that $N_G[\widetilde{X}_j] \subseteq N_G[X_j] \subseteq X_j \cup Y_j = \widetilde{X}_j \cup \widetilde{Y}_j$. Consequently, $(\widetilde{X}_1,\widetilde{Y}_1),\dots,(\widetilde{X}_t,\widetilde{Y}_t)$ are chunks of $G$, and the sets $\widetilde{X}_1,\dots,\widetilde{X}_t$ are pairwise disjoint. 

Suppose now that $\sum_{j=1}^t p(X_j,Y_j) \geq 6$. We need to show that at least one of (i)-(vii) holds. We may assume that $C \neq \emptyset$, for otherwise, (i) trivially holds. 

Suppose first that $p(X_j,Y_j) = 6$ for some $j \in \{1,\dots,t\}$; by symmetry, we may assume that $p(X_1,Y_1) = 6$. Then $|X_1| \geq 5$ and $|Y_1| = 0$, and so $|\widetilde{X}_1| \geq 3$ and $|\widetilde{Y}_1| \leq 2$. If $|\widetilde{X}_1| \geq 4$, then we see from Table~\ref{t:wc} that $p(\widetilde{X}_1,\widetilde{Y}_1) \geq 4$, and so (i) holds, and we are done. So assume that $|\widetilde{X}_1| = 3$. Since $|X_1| \geq 5$, this implies that $|X_1| = 5$ and $|C| = 2$. It follows that $|\widetilde{Y}_1| = 2$, and since $|\widetilde{X}_1| = 3$, we see from Table~\ref{t:wc} that $p(\widetilde{X}_1,\widetilde{Y}_1) = 3$. If the two vertices of $C$ are adjacent, then (iv) holds, and if they are non-adjacent, then outcome (v) holds, and in either case, we are done. 

From now on, we assume that $p(X_j,Y_j) \leq 4$ for all $j \in \{1,\dots,t\}$. Suppose first that $C$ is included in one of $X_1,\dots,X_t$; by symmetry, we may assume that $C \subseteq X_1$. Then $(\widetilde{X}_j,\widetilde{Y}_j) = (X_j,Y_j)$ for all $j \in \{2,\dots,t\}$. If $p(X_1,Y_1) \leq 2$, then $\sum_{j=2}^t p(\widetilde{X}_j,\widetilde{Y}_j) \geq 4$, and (i) holds. Next, if $p(X_1,Y_1) = 2.5$, then $|X_1| = 2$ and $|Y_1| \leq 2$, and so every vertex in $X_1$ (and in particular, every vertex in $C$) is of degree at most $3$, and we deduce that (ii) holds. Finally, if $3 \leq p(X_1,Y_1) \leq 4$, then using the fact that $|\widetilde{X}_1| \geq |X_1|-2$ and $|\widetilde{Y}_1| \leq |Y_1|+2$, we deduce from Table~\ref{t:wc} that $p(\widetilde{X}_j,\widetilde{Y}_j) \geq p(X_1,Y_1)-2$, and it follows that (i) holds. 

It remains to consider the case when $C$ is not included in any one of $X_1,\dots,X_t$. Then $|C| = 2$, and we may assume by symmetry that one vertex of $C$ belongs to $X_1$ and the other to $X_2$. Then for each $j \in \{1,2\}$, we have that $|\widetilde{X}_j| = |X_j|-1$ and $|\widetilde{Y}_j| = |Y_j|+1$, and we see from Table~\ref{t:wc} that $p(\widetilde{X}_j,\widetilde{Y}_j) \geq p(X_j,Y_j)-2$. Furthermore, for all $j \in \{3,\dots,t\}$, we have that  $(\widetilde{X}_j,\widetilde{Y}_j) = (X_j,Y_j)$. Now, by symmetry, it suffices to consider the following three cases: 
\begin{enumerate}[(a)] 
\item $p(X_j,Y_j)-1 \leq p(\widetilde{X}_j,\widetilde{Y}_j)$ for each $j \in \{1,2\}$; 
\item $p(X_1,Y_1)-2 \leq p(\widetilde{X}_1,\widetilde{Y}_1) < p(X_1,Y_1)-1$ and $p(X_2,Y_2)-1 \leq p(\widetilde{X}_2,\widetilde{Y}_2)$; 
\item $p(X_j,Y_j)-2 \leq p(\widetilde{X}_j,\widetilde{Y}_j) < p(X_j,Y_j)-1$ for each $j \in \{1,2\}$. 
\end{enumerate} 

If (a) is true, then (i) holds, and we are done. Suppose next that (b) is true. Then we deduce from Table~\ref{t:wc} that $|X_1| = 1$ and $|Y_1| \leq 3$. In particular then, the unique vertex of $X_1$ belongs to $C$, and the degree of this vertex is at most $|Y_1|$. If $|Y_1| = 3$, then outcome (ii) holds, and if $|Y_1| \leq 2$, then outcome (iii) holds, and in either case, we are done. It remains to consider the case when (c) holds. Then we see from Table~\ref{t:wc} that for each $j \in \{1,2\}$, $|X_j| = 1$ and $|Y_j| \leq 3$, and in particular, the unique vertex of $X_j$ belongs to $C$, and the degree of this vertex is at most $|Y_j|$. If $|Y_1| = |Y_2| = 3$, then outcome (v) holds; if $|Y_j| \leq 2$ and $|Y_{3-j}| = 3$ for some $j \in \{1,2\}$, then outcome (vi) holds; and if $|Y_j| \leq 2$ for each $j \in \{1,2\}$, then outcome (vii) holds. This completes the argument. 
\end{proof}

\begin{lemma} \label{chunk} Let $G$ be a graph on at least three vertices. Then there exist chunks $(X_1,Y_1),\dots,(X_t,Y_t)$ of $G$ that satisfy the following three properties: 
\begin{itemize} 
\item[(a)] $X_1,\dots,X_t$ are pairwise disjoint; 
\item[(b)] for all $j \in \{1,\dots,t\}$, if $|X_j| \geq 2$, then $G[X_j \cup Y_j]$ is $3$-connected; 
\item[(c)] $\sum_{j=1}^t p(X_j,Y_j) \geq 6$. 
\end{itemize} 
\end{lemma} 
\begin{proof} 
We assume inductively that the claim holds for all graphs $G'$ such that $3 \leq |V(G')| < V(G)$. Suppose first that $3 \leq |V(G)| \leq 4$. We then set $t = |V(G)|$, we let $(X_1,\dots,X_t)$ be a partition of $V(G)$ into sets of size one, and for each $j \in \{1,\dots,t\}$, we set $Y_j = V(G) \smallsetminus X_j$. Clearly, $(X_1,Y_1),\dots,(X_t,Y_t)$ are chunks of $G$, they satisfy (a) and (b) by construction, and we see from Table~\ref{t:wc} that they satisfy (c). From now on, we assume that $|V(G)| \geq 5$. If $G$ is $3$-connected, then we set $t = 1$, $X_1 = V(G)$, $Y_1 = \emptyset$, we observe that $(X_1,Y_1)$ satisfies (a) and (b) by construction, and we see from Table~\ref{t:wc} that $p(X_1,Y_1) = 6$, and so $(X_1,Y_1)$ satisfies (c) as well. From now on, we assume that $G$ is not $3$-connected. Since $|V(G)| \geq 5$, we easily deduce that $G$ admits a cutset $C$ of size exactly two. Thus, there exist graphs $G_1$ and $G_2$ such that $3 \leq |V(G_1)|,|V(G_2)| < |V(G)|$, and such that $G$ is obtained by gluing $G_1$ and $G_2$ along $C$. Using the induction hypothesis, for each $i \in \{1,2\}$, we fix chunks $(X_1^i,Y_1^i),\dots,(X_{t_i}^i,Y_{t_i}^i)$ of $G_i$ such that the following hold: 
\begin{itemize} 
\item the sets $X_1^i,\dots,X_{t_i}^i$ are pairwise disjoint; 
\item for all $j \in \{1,\dots,t_i\}$, if $|X_j^i| \geq 2$, then $G_i[X_j^i \cup Y_j^i]$ is $3$-connected; 
\item $\sum_{j=1}^{t_i} p(X_j^i,Y_j^i) \geq 6$. 
\end{itemize} 
For all $i \in \{1,2\}$ and $j \in \{1,\dots,t_i\}$, set $\widetilde{X}_j^i = X_j^i \smallsetminus C$ and $\widetilde{Y}_j^i = Y_j^i \cup (X_j^i \cap C)$. Clearly, for all $i \in \{1,2\}$ and $j \in \{1,\dots,t_i\}$, we have that $\widetilde{X}_j^i \cup \widetilde{Y}_j^i = X_j^i \cup Y_j^i$ and $G[\widetilde{X}_j^i \cup \widetilde{Y}_j^i] = G_i[X_j^i \cup Y_j^i]$, and furthermore, since $V(G_1) \smallsetminus C$ is anti-complete to $V(G_2) \smallsetminus C$ in $G$, we know that $N_G[\widetilde{X}_j^i] = N_{G_i}[\widetilde{X}_j^i] \subseteq N_{G_i}[X_j^i] \subseteq X_j^i \cup Y_j^i = \widetilde{X}_j^i \cup \widetilde{Y}_j^i$. It follows that for all $i \in \{1,2\}$ and $j \in \{1,\dots,t_i\}$, $(\widetilde{X}_j^i,\widetilde{Y}_j^i)$ is a chunk of both $G$ and $G_i$. Next, recall that $|C| = 2$, and set $C = \{c_1,c_2\}$, $C_1 = N_G(c_1)$, and $C_2 = N_G(c_2)$; then $(\{c_1\},C_1)$ and $(\{c_2\},C_2)$ are chunks of $G$. Clearly, the chunks $(\{c_1\},C_1),(\{c_2\},C_2),(\widetilde{X}_1^1,\widetilde{Y}_1^1),\dots,(\widetilde{X}_{t_1}^1,\widetilde{Y}_{t_1}^1),(\widetilde{X}_1^2,\widetilde{Y}_1^2),\dots,(\widetilde{X}_{t_2}^2,\widetilde{Y}_{t_2}^2)$ of $G$ satisfy (a) and (b). It remains to show that they satisfy (c), that is, that $p(\{c_1\},C_1)+p(\{c_2\},C_2)+\sum_{i=1}^2\sum_{j=1}^{t_i} p(\widetilde{X}_j^i,\widetilde{Y}_j^i) \geq 6$. 

For each $i \in \{1,2\}$, we apply Lemma~\ref{delete-C} to the graph $G_i$, the chunks $(\widetilde{X}_1^i,\widetilde{Y}_1^i),\dots,(\widetilde{X}_{t_i}^i,\widetilde{Y}_{t_i}^i)$ of $G_i$, and the set $C$. By symmetry, we may assume that one of the following holds: 
\begin{itemize} 
\item[(1)] $G_1,(\widetilde{X}_1^1,\widetilde{Y}_1^1),\dots,(\widetilde{X}_{t_1}^1,\widetilde{Y}_{t_1}^1),C$ satisfy (i), and \\ $G_2,(\widetilde{X}_1^2,\widetilde{Y}_1^2),\dots,(\widetilde{X}_{t_2}^2,\widetilde{Y}_{t_2}^2),C$ satisfy one of (i)-(vii); 
\item[(2)] for each $i \in \{1,2\}$, $G_i,(\widetilde{X}_1^i,\widetilde{Y}_1^i),\dots,(\widetilde{X}_{t_i}^i,\widetilde{Y}_{t_i}^i),C$ satisfy one of (ii)-(v); 
\item[(3)] $G_1,(\widetilde{X}_1^1,\widetilde{Y}_1^1),\dots,(\widetilde{X}_{t_1}^1,\widetilde{Y}_{t_1}^1),C$ satisfy (ii), and \\ $G_2,(\widetilde{X}_1^2,\widetilde{Y}_1^2),\dots,(\widetilde{X}_{t_2}^2,\widetilde{Y}_{t_2}^2),C$ satisfy (vi); 
\item[(4)] $G_1,(\widetilde{X}_1^1,\widetilde{Y}_1^1),\dots,(\widetilde{X}_{t_1}^1,\widetilde{Y}_{t_1}^1),C$ satisfy (ii), and \\ $G_2(\widetilde{X}_1^2,\widetilde{Y}_1^2),\dots,(\widetilde{X}_{t_2}^2,\widetilde{Y}_{t_2}^2),C$ satisfy (vii); 
\item[(5)] $G_1,(\widetilde{X}_1^1,\widetilde{Y}_1^1),\dots,(\widetilde{X}_{t_1}^1,\widetilde{Y}_{t_1}^1),C$ satisfy one of (iii)-(v), and \\ $G_2,(\widetilde{X}_1^2,\widetilde{Y}_1^2),\dots,(\widetilde{X}_{t_2}^2,\widetilde{Y}_{t_2}^2),C$ satisfy (vi) or (vii); 
\item[(6)] for each $i \in \{1,2\}$, $G_i,(\widetilde{X}_1^i,\widetilde{Y}_1^i),\dots,(\widetilde{X}_{t_i}^i,\widetilde{Y}_{t_i}^i),C$ satisfy (vi); 
\item[(7)] $G_1,(\widetilde{X}_1^1,\widetilde{Y}_1^1),\dots,(\widetilde{X}_{t_1}^1,\widetilde{Y}_{t_1}^1),C$ satisfy (vi) or (vii), and \\ $G_2,(\widetilde{X}_1^2,\widetilde{Y}_1^2),\dots,(\widetilde{X}_{t_2}^2,\widetilde{Y}_{t_2}^2),C$ satisfy (vii). 
\end{itemize} 
If (1), (2), or (3) holds, then $\sum_{i=1}^2\sum_{j=1}^{t_i} p(\widetilde{X}_j^i,\widetilde{Y}_j^i) \geq 6$, and we are done. If (4), (5), or (6) holds, then $\sum_{i=1}^2\sum_{j=1}^{t_i} p(\widetilde{X}_j^i,\widetilde{Y}_j^i) \geq 5$, and we see by routine checking that $C$ contains a vertex of degree at most $5$ in $G$; we then deduce from Table~\ref{t:wc} that either $p(\{c_1\},C_1) \geq 1$ or $p(\{c_2\},C_2) \geq 1$, and it follows that $p(\{c_1\},C_1)+p(\{c_2\},C_2)+\sum_{i=1}^2\sum_{j=1}^{t_i} p(\widetilde{X}_j^i,\widetilde{Y}_j^i) \geq 6$. It remains to consider the case when (7) holds. Then $\sum_{i=1}^2\sum_{j=1}^{t_i} p(\widetilde{X}_j^i,\widetilde{Y}_j^i) \geq 4$. Furthermore, ${\rm deg}_G(c_1),{\rm deg}_G(c_2) \leq 5$, and we deduce from Table~\ref{t:wc} that $p(\{c_1\},C_1),p(\{c_2\},C_2) \geq 1$. It follows that $p(\{c_1\},C_1)+p(\{c_2\},C_2)+\sum_{i=1}^2\sum_{j=1}^{t_i} p(\widetilde{X}_j^i,\widetilde{Y}_j^i) \geq 6$. This completes the argument. 
\end{proof}

We are now ready to prove that $\psi_c(2) \leq 5$. Let $G$ be a graph. By the definition of $\psi_c$, we need to show that at least one of the following holds:
\begin{enumerate}[(i)] 
\item $G$ is $3$-connected; 
\item $G$ admits a cut-partition $(A,B,C)$ such that $G[A \cup C]$ is $3$-connected and $|C| \leq 5$; 
\item $G$ contains a vertex of degree at most $5$. 
\end{enumerate} 
We may assume that $|V(G)| \geq 3$, for otherwise (iii) holds, and we are done. Lemma~\ref{chunk} then guarantees that there exists a chunk $(X,Y)$ of $G$ such that the following hold: 
\begin{itemize} 
\item $p(X,Y) > 0$; 
\item if $|X| \geq 2$, then $G[X \cup Y]$ is $3$-connected. 
\end{itemize} 
Since $(X,Y)$ is a chunk of $G$, we know that $N_G(X) \subseteq Y$. Since $p(X,Y) > 0$, we deduce from Table~\ref{t:wc} that $X \neq \emptyset$ and $|Y| \leq 5$. If $|X| = 1$, then the fact that $|N_G(X)| \leq |Y| \leq 5$ implies that the unique vertex of $X$ is of degree at most $5$, and so (iii) holds. So assume that $|X| \geq 2$. Then $G[X \cup Y]$ is $3$-connected. If $X \cup Y = V(G)$, then $G$ is $3$-connected, and (i) holds. Otherwise, we set $A = X$, $B = V(G) \smallsetminus (X \cup Y)$, and $C = Y$. Since $N_G(X) \subseteq Y$ and $X \cup Y \subsetneq V(G)$, we see that $(A,B,C)$ is a cut-partition of $G$. Since $G[A \cup C] = G[X \cup Y]$ is $3$-connected and $|C| = |Y| \leq 5$, it follows that (ii) holds. This proves that $\psi_c(2) \leq 5$.

\section{Proof of Theorem~\ref{th:col}}
\label{sec:col}

The main goal of this section is to prove Theorem~\ref{th:col}. In this section, it is sometimes useful to allow graphs to be empty, and so we no longer implicitly assume that our graphs are non-null. We remark that a coloring of a graph can be seen as a partition of the vertex-set of that graph into stable sets (``color classes''); in particular then, for every graph $G$, $V(G)$ can be partitioned into $\chi(G)$ stable sets. We now prove a lemma, which we will use twice in the proof of Theorem~\ref{th:col}.

\begin{lemma} 
 \label{lemma:col} 
 Let $G$ be a graph, let $(X,Y)$ be a partition of $V(G)$ into two
 (possibly empty) sets, and let $q$ be the number of edges of $G$
 between $X$ and $Y$. Then $\chi(G) \leq
 \max\{\chi(G[X]),\chi(G[Y]),q+1\}$.
\end{lemma} 
\begin{proof} 
 We may assume that both $X$ and $Y$ are non-empty, for otherwise the result is immediate. 
 We set $s = \chi(G[X])$ and $t = \chi(G[Y])$. Let $X_1, \dots, X_s$ be
 disjoint stable sets that partition $X$, and let $Y_1, \dots, Y_t$ be
 disjoint stable sets that partition $Y$. We now build an auxiliary
 bipartite graph $H$. We set $V(H) = \{X_1, \dots, X_s, Y_1, \dots, Y_t\}$,
 and we link $X_i$ to $Y_j$ in $H$ when $X_i$ is anti-complete to $Y_j$ in $G$. 
 These are the only edges of $H$, so $H$ is bipartite because 
 $\{X_1,\dots,X_s\}$ and $\{Y_1,\dots,Y_t\}$ are stable sets in $H$. 

 We now apply a classical result of K\H onig, which guarantees that in our
 bipartite graph $H$, there exists a matching $M$ and a set $Z$ of
 vertices such that $Z$ meets all edges of $H$ and $|M| = |Z|$. We set
 $x= |Z \cap \{X_1, \dots, X_s\}|$ and $y = |Z \cap \{Y_1, \dots,
 Y_t\}|$. Clearly, every edge in $M$ meets exactly one vertex of
 $Z$, and by symmetry, we may assume that $M = \{X_1Y_1, \dots,
 X_{x+y}Y_{x+y}\}$ and that $Z = \{X_1, \dots, X_x, Y_{x+1}, \dots,
 Y_{x+y}\}$.

 We now color $G$ as follows. For all $i\in \{1, \dots, x+y\}$, we
 give color $i$ to vertices of $X_i \cup Y_i$. In each of the remaing
 sets, namely $X_{|M|+1}, \dots, X_s$ and $Y_{|M|+1}, \dots, Y_t$, we
 color all vertices with the same color, but for each set, we
 choose a color not previously used. This is a proper coloring of $G$,
 because the $X_i$'s and $Y_j$'s are stable sets of $G$, and from the
 definitions of $H$ and $M$, $X_i$ is anti-complete to $Y_i$ in $G$ when 
 $1 \leq i \leq x+y$. By counting the number of colors that we used, we obtain
 $$ \chi(G) \leq |M| + (s- |M|) + (t - |M|) = s + t - |M| =
 (s-x) + (t-y).$$

If $x=s$, then $M$ covers $\{X_1, \dots, X_s\}$, and it follows that $y=0$ and 
$\chi(G) \leq (s-x) + (t-y) = t$. By an analogous argument, if $y=t$, then 
 $\chi(G) \leq s$. 

It remains to consider the case when $x<s$ and $y<t$. We then have that 
$s-x \geq 1$ and $t-y \geq 1$, and consequently, 
$(s-x) + (t-y) \leq (s-x)(t-y) + 1$. 
From the definition of $Z$, we know that for all $i \in
 \{x+1, \dots, s\}$ and all $j \in \{1, \dots, x\} \cup \{x+y+1,
 \dots, t\}$, $X_i$ and $Y_j$ are not linked by an edge of $H$. This
 means that there exists an edge of $G$ between $X_i$ and
 $Y_j$. Hence, the number $q$ of edges of $G$ is at least 
 $$ |\{x+1, \dots, s\}| \cdot |\{1, \dots, x\} \cup \{x+y+1,
 \dots, t\}| = (s-x)(t-y).$$
 So we have $$\chi(G) \leq (s-x) + (t-y) \leq (s-x)(t-y) +
 1 \leq q+1.$$ 
 
 From the two preceding paragraphs, we see that $\chi(G)$ is bounded
 above by either $s$, $t$, or $q+1$. The result follows.
\end{proof} 

\noindent 
An {\em optimal coloring} of a graph $G$ is a proper coloring of $G$
that uses exactly $\chi(G)$ colors. A graph $G$ is said to be {\em
critical} if for all non-empty sets $S \subseteq
V(G)$, we have that $\chi(G \smallsetminus S) < \chi(G)$. We are now
ready to prove Theorem \ref{th:col}, restated below. 

\begin{th:col} 
 Let $k$ be a positive and $c$ a non-negative integer, and let $G$ be
 a graph such that $\chi(G) > \max\{c+2k-2,2k^2\}$. Then $G$ contains
 a $(k+1)$-connected induced subgraph of chromatic number greater
 than $c$.
\end{th:col} 
\begin{proof} 
 Set $M = \max\{c+2k-2,2k^2\}$. We may assume without loss of
 generality that $\chi(G) = M+1$, and that $G$ is 
 critical. We may also assume that $G$ is not $(k+1)$-connected, for
 otherwise, $G$ itself is a $(k+1)$-connected induced subgraph of $G$
 of chromatic number greater than $c$, and we are done. Since $G$ is
 critical, $\delta(G) \geq \chi(G)-1 \geq
 2k^2$. Corollary~\ref{cor-main} now guarantees that $G$ admits a
 cut-partition $(A,B,C)$ such that $G[A \cup C]$ is $(k+1)$-connected
 and $w_B^k(C) \leq 2k^2-1$. Since $G$ is critical, we have that
 $\chi(G[B \cup C]) \leq M$. Now, we claim that $\chi(G[A \cup
 C]) > c$. Suppose otherwise, that is, suppose that $\chi(G[A \cup
 C]) \leq c$. Our goal is to use Lemma~\ref{lemma:col} to prove that
 $\chi(G) \leq M$, contrary to the fact that $\chi(G) = M+1$.

 Let $C_S$ be the set of all vertices in $C$ that are $k$-strong with
 respect to $B$, and let $C_W$ be the set of all vertices in $C$ that
 are $k$-weak with respect to $B$; thus, $C = C_S \cup C_W$ and $C_S
 \cap C_W = \emptyset$. Consider any optimal coloring of $G[B \cup
 C_S]$, let $S_1,\dots,S_r$ be the list of all color classes of this
 coloring that intersect $C_S$, and set $S = \bigcup_{j=1}^r
 S_j$. (If $C_S = \emptyset$, then we simply have that $r = 0$ and $S
 = \emptyset$.) Clearly, $C_S \subseteq S$, $r \leq |C_S|$,
 $\chi(G[S]) = r$, and $\chi(G[B \smallsetminus S]) = \chi(G[B \cup
 C_S])-r \leq M-r$. Since every vertex in $C_S$ is $k$-strong with
 respect to $B$, we have that $w_B(C_S) = k|C_S| \geq kr$, and since
 every vertex in $C_W$ is $k$-weak with respect to $B$, the number of
 edges between $C_W$ and $B$ is at most $w_B(C_W) = w_B(C)-w_B(C_S)
 \leq 2k^2-kr-1$. Since $A$ is anti-complete to $B$, it follows that
 the number of edges between $A \cup C_W$ and $B$ is at most
 $2k^2-kr-1 \leq M-r-1$.

 Since $r \leq |C_S|$ and $k|C_S| = w_B(C_S) \leq 2k^2-1$, we have
 that $r \leq |C_S| \leq 2k-1$. Suppose first that $r \leq 2k-2$. Then
 $\chi(G[A \cup C_W]) \leq c \leq M-r$. Since $\chi(G[B
 \smallsetminus S]) \leq M-r$, and since the number of edges between
 $A \cup C_W$ and $B$ is at most $M-r-1$, Lemma~\ref{lemma:col}
 implies that $\chi(G[(A \cup C_W) \cup (B \smallsetminus S)] \leq
 M-r$. Since $C_S \subseteq S$ and $\chi(G[S]) = r$, it follows that
 $\chi(G) \leq M$, contrary to the fact that $\chi(G) = M+1$.

 It remains to consider the case when $r = 2k-1$. Then $|C_S| =
 2k-1$. Since every set among $S_1,\dots,S_r$ intersects $C_S$, we
 deduce that $|S_j \cap C_S| = 1$ for all $j \in
 \{1,\dots,r\}$. Further, the number of edges between $A \cup C_W$
 and $B$ is at most $2k^2-kr-1 = k-1 < r$, and so $A \cup C_W$ is
 anti-complete to at least one of the sets $S_1 \smallsetminus
 C_S,\dots,S_r \smallsetminus C_S$, say to $S_r \smallsetminus
 C_S$. Let $c_S$ be the unique member of $S_r \cap C_S$. Consider any
 optimal coloring of $G[A \cup C_W \cup \{c_S\}]$, and let $S'$ be
 the color class of this coloring that contains $c_S$. Then $S' \cup
 S_r$ is a stable set, and it follows that $\chi(G[S \cup S']) =
 r$. On the other hand, $\chi(G[(A \cup C_W) \smallsetminus S']) =
 \chi(G[A \cup C_W])-1 \leq c-1 \leq M-r$. Recall that $\chi(G[B
 \smallsetminus S]) \leq M-r$, and that the number of edges between
 $A \cup C_W$ and $B$ is at most $M-r-1$. Lemma~\ref{lemma:col} now
 implies that $\chi(G[((A \cup C_W) \smallsetminus S') \cup (B
 \smallsetminus S)]) \leq M-r$. This, together with the fact that
 $C_S \subseteq S$ and $\chi(G[S \cup S']) \leq r$, implies that
 $\chi(G) \leq M$, contrary to the fact that $\chi(G) = M+1$. This
 completes the argument.
\end{proof} 

The following proposition was announced in Section~\ref{sec:main}. We 
remark that Theorem~\ref{chi-new} (which we derived from 
Theorem~\ref{th:col} in a more direct fashion in Section~\ref{sec:main}) 
can in fact be obtained as an immediate corollary of Theorem~\ref{th:col} 
and Proposition~\ref{chi-conn-class-equiv} by setting $h(n) = 
\max\{n+2k-2,2k^2\}$. 

\begin{proposition} \label{chi-conn-class-equiv} Let $k$ be a positive
 integer, and let $h:\mathbb{N} \rightarrow \mathbb{N}$ be a
 non-decreasing function. Then the following two statements are
 equivalent:
\begin{enumerate}[(i)] 
\item for all non-negative integers $c$, and all graphs $G$ such that
 $\chi(G) > h(c)$, $G$ contains an induced subgraph $H$ such that $H$
 contains no cutset of size at most $k$, and $\chi(H) > c$;
\item for all non-decreasing functions $f:\mathbb{N} \rightarrow
 \mathbb{N}$, and all hereditary classes $\mathcal{G}$,
 $\chi$-bounded by the function $f$, the class $\mathcal{G}^k$ is
 $\chi$-bounded by $h \circ f$.
\end{enumerate} 
\end{proposition} 
\begin{proof} 
Suppose first that the function $h$ satisfies (i). To show that $h$ satisfies (ii), fix a non-decreasing function $f:\mathbb{N} \rightarrow \mathbb{N}$, a hereditary class $\mathcal{G}$, $\chi$-bounded by $f$, and a graph $G \in \mathcal{G}^k$. We need to show that $\chi(G) \leq h \circ f(\omega(G))$. Suppose otherwise. Then $\chi(G) > h(f(\omega(G)))$, and so by (i), $G$ contains an induced subgraph $H$ such that $H$ contains no cutset of size at most $k$, and $\chi(H) > f(\omega(G))$. Since $H$ is an induced subgraph of $G$, we know that $\omega(H) \leq \omega(G)$, and since $f$ is non-decreasing, this implies that $\chi(H) > f(\omega(H))$. Since $\mathcal{G}$ is hereditary, so is $\mathcal{G}^k$. Since $G \in \mathcal{G}^k$, and $H$ is an induced subgraph of $G$, we know that $H \in \mathcal{G}^k$. Since $H$ contains no cutset of size at most $k$, Lemma~\ref{l:inG} implies that $H \in \mathcal{G}$. Since $\mathcal{G}$ is $\chi$-bounded by $f$, it follows that $\chi(H) \leq f(\omega(H))$, which is a contradiction. Thus, $h$ satisfies (ii). 

Suppose now that the function $h$ satisfies (ii). To show that $h$ satisfies (i), fix a non-negative integer $c$ and a graph $G$ such that $\chi(G) > h(c)$. Let $\mathcal{G}$ be the class of all graphs whose chromatic number is at most $c$. Then by (ii), the chromatic number of any graph in $\mathcal{G}^k$ is at most $h(c)$, and consequently, $G \notin \mathcal{G}^k$. Let $H$ be a minimal induced subgraph of $G$ such that $H \notin \mathcal{G}^k$. Since $\mathcal{G}^k$ is closed under the operation of gluing along at most $k$ vertices, the minimality of $H$ guarantees that $H$ does not admit a cutset of size at most $k$. Since $H \notin \mathcal{G}^k$, we know $H \notin \mathcal{G}$, and so by the definition of $\mathcal{G}$, $\chi(H) > c$. This proves that $h$ satisfies (i). 
\end{proof}

\end{document}